\documentclass{article}
\usepackage[a4paper,hmargin=1.1in,vmargin=1.1in]{geometry}
\usepackage[T1]{fontenc}
\usepackage{amsmath}
\usepackage{amsthm}
\usepackage{amssymb}
\usepackage{amsfonts}
\usepackage{enumerate}
\usepackage{hyperref}

\newcommand{\N}{\ensuremath{\mathbb{N}}}

\newcommand{\FF}{\ensuremath{\mathbb{F}}}

\newcommand {\Fq}{\FF_q}

\newcommand{\out}[1]{}

\def\CC{\mathbb{C}}

\def\ZZ{\mathbb{Z}}
\def\PP{\mathbb{P}}

\def \mA {\mathcal{A}}

\def \mA {\mathcal{A}}

\def \mF {\mathcal{F}}
\def \mE {\mathcal{E}}
\def \mI {\mathcal{I}}
\def \mJ {\mathcal{J}}
\def \L {\mathcal{L}}
\def \mL {\mathcal{L}}

\def\cL{{\mathcal L}}

\def \cX {\mathcal{X}}
\def \J {\mathcal{J}}

\def \ba {{\bf a}}
\def \bc {{\bf c}}
\def \bx {{\bf x}}

\def\supp {{\rm supp }}
\def \cc {{\bf c}}
\def\F{{\mathbb{F}}}
\newcommand {\divi}[1]{\Div_{#1}(F)}
\def\Div{{\rm Div}}
\def\Prin{{\rm Prin}}
\def\Cl{{\rm Cl}}
\def\Exp{{\rm Exp}}

\def\d{{\rm desc}}

\newcommand{\pict}[2]{{\mathcal J}[#2]}

\def\BbbZ{{\sf Z\hspace*{-0.95ex}Z}}
\newcommand{\Z}[1]{\relax\ifmmode\BbbZ_{#1}\else $\BbbZ_{#1}$\fi}

\newcommand{\cancel}[1]{}

\newcommand{\Proj}{\mathbb{P}}

\newtheorem{theorem}{Theorem}[section]
\newtheorem{definition}[theorem]{Definition}
\newtheorem{proposition}[theorem]{Proposition}
\newtheorem{lemma}[theorem]{Lemma}
\newtheorem{corollary}[theorem]{Corollary}
\numberwithin{equation}{section} 

\newtheorem{remark}[theorem]{Remark}

\title{Torsion Limits and Riemann-Roch Systems
for Function Fields and Applications\footnote{Version accepted for publication in\emph{IEEE Transactions on Information Theory}. DOI: 10.1109/TIT.2014.2314099, 
URL (early access version): \url{http://ieeexplore.ieee.org/stamp/stamp.jsp?tp=&arnumber=6779612}. 
Copyright (c) 2012 IEEE. Personal use of this material is permitted. Permission from
IEEE must be obtained for all other uses, in any current or future media, including reprinting/republishing this material
for advertising or promotional purposes, creating new collective works, for resale or redistribution to servers or lists, or
reuse of any copyrighted component of this work in other works.
This is an extended version of our paper~\cite{CCX11}  in
Proceedings of 31st Annual IACR CRYPTO, Santa
Barbara, Ca., USA, 2011. The results in Sections~\ref{sec:mult} and~\ref{sec:frame}
did not appear in \cite{CCX11}. A first version of this paper  has been widely circulated since November 2009.
I. Cascudo was supported in part by the STW Sentinels program under Project 10532 and in part by Cramer's NWO VICI Grant ``Mathematics of Secure Computation''. R. Cramer was supported in part by NWO VICI Grant
``Mathematics of Secure Computation.'' C. Xing was supported by the Singapore Minister of Education under Tier 1 grant RG20/13.}
}

\author{Ignacio Cascudo\thanks{CWI Amsterdam, The Netherlands (at the time of this research; currently at Aarhus University, Denmark). Email: \texttt{ignacio@cs.au.dk}.} \and Ronald Cramer\thanks{CWI Amsterdam \& Mathematical Institute, Leiden University, The Netherlands.
Email: \texttt{cramer@cwi.nl, cramer@math.leidenuniv.nl}.} \and Chaoping Xing\thanks{Division of Mathematical Sciences, Nanyang Technological University, Singapore. Email: \texttt{xingcp@ntu.edu.sg}.} }

\date{}

\begin{document}
\maketitle
\begin{abstract}
The Ihara limit (or constant) $A(q)$  has been a central problem of study in the asymptotic theory of global function fields  (or equivalently,  algebraic curves over finite fields). It addresses global function fields  with many rational points and,
 so far, most applications of this theory do not
require additional properties. Motivated by recent applications, we require global function fields
with the additional property that their zero class divisor groups contain at most a small number of $d$-torsion points. We capture this with the notion of torsion limit, a new asymptotic quantity for global function fields.
 It seems that it is even harder to determine values of this new quantity than the Ihara constant.
 Nevertheless, some non-trivial upper bounds are derived.
 Apart from this new asymptotic quantity and bounds on it, we also introduce Riemann-Roch systems of equations. It turns out that this type of equation system
 plays an important role in the study of several other problems in each of these areas: arithmetic secret sharing, symmetric bilinear complexity of multiplication in finite fields, frameproof codes and the theory of error correcting codes.
 Finally, we show how our new asymptotic quantity, our bounds on it and Riemann-Roch systems can be used to improve results in these areas.

\end{abstract}

\section{Introduction}

Since the discovery of algebraic geometry codes by Goppa~\cite{G81} and other
applications such as low-discrepancy sequences \cite{NX96}, the study of
algebraic curves with many rational points over finite fields or,
equivalently, global function fields with many rational places, has
attracted many researchers from various areas, such as pure
mathematicians, coding theorists and algorithmically inclined
mathematicians. In the last two decades, there have been tremendous
research activities in this topic.

A crucial quantity in the asymptotic theory of
global function fields with many rational places, namely the Ihara limit, plays an important
role in coding theory and other topics. Precisely speaking, for a
given prime power $q$, the Ihara limit is defined by
\[A(q):=\limsup_{g\rightarrow\infty}\frac{N_q(g)}{g},\]
where $N_q(g)$ denotes the maximum number of rational places taken over all global function fields over $\FF_q$ of genus $g$.

The Drinfeld-Vl\v{a}du\c{t} bound states that $A(q)\leq \sqrt{q}-1$.
By Ihara~\cite{Ihara}, $A(q)=\sqrt{q}-1$ if $q$ is a square. By
Serre's Theorem~\cite{Serre85}, $A(q)\geq c\cdot \log q$ for some
 absolute real constant $c>0$
 (for which the current best lower bound~\cite{NX01} is approximately $\frac{1}{96}$).

So far, most applications of global function fields do not
require additional properties. Motivated by recent applications (arithmetic secret sharing, see below), we require global function fields
with the additional property that their zero  divisor class groups contain at most a small number of $d$-torsion points. The exact same requirements are
relevant for the study of the symmetric bilinear complexity of multiplication in finite fields. Although the latter topic started much earlier, the role of $2$-torsion points in the zero class
divisor groups was
overlooked in~\cite{STV92,B08}.

In this paper, we introduce
two new primitives for function fields over finite fields, namely the
torsion limit and systems of Riemann-Roch equations.  Our torsion
limit, which we believe is of independent interest, can in general
be upper bounded using Weil's classical theorem on torsion in
Abelian varieties (and in many cases using the Weil-pairing).
\footnote{We note that, independently, Randriambololona~\cite{Randriam10} introduced the same notion of torsion limit (for optimal families of function fields) in the context of an application to the construction of frameproof codes and proved the bounds
that follow directly from Weil's classical result.}
However, the resulting bound is far too pessimistic, as we present a
tower for which our torsion limit is {\em considerably smaller}, yet
it attains the Drinfeld-Vl\v{a}du\c{t} bound.

A system of Riemann-Roch equations consists of simultaneous equations whose variables are divisors.
 Although Riemann-Roch systems have been implicitly studied in coding theory~\cite{V87,X01,CC02,X05,Xu05,M07,NO07}
 such a concept has not been formally introduced. Moreover, we are interested in systems of a more general type than the ones considered in those papers,
 as we will explain. In several interesting cases, the existence of solutions will depend very much on the torsion in the class group. Hence, in the asymptotic case, where we consider Riemann-Roch systems in a tower of function fields, its solvability
will depend on our new torsion limit.

We give three applications in this paper that demonstrate the importance of such systems, in conjunction with our torsion limit and bounds on it.
First, arithmetic secret sharing schemes are a special kind of codes arising in secure multi-party computation~\cite{CDM00,CC06}. 
Using optimal towers of function fields, Chen and Cramer~\cite{CC06} showed the existence of ``asymptotically good'' families of such
schemes. Since then, the asymptotical results of~\cite{CC06} have had several important and surprising applications in {\em two-party} cryptography~\cite{IKOS07,IPS08,HIKN08,IKOS09,DIK10,IKOPSW11}. 
The results of~\cite{CC06} were improved and extended in~\cite{CCHP08,CCCX09}. We show how our torsion limits and Riemann-Roch
equations allow to further improve those results.

In fact, the arguments from~\cite{CC06} also show the existence of linear codes such that both the duals and ``powers'' are
simultaneously asymptotically good, where we define the $d$-th power $C^{*d}$ of a linear code $C$
to be the linear code spanned by all possible coordinate-wise products of $d$ (not necessarily distinct) words in $C$. 
The results in~\cite{CC06} imply that for any fixed integer $d\geq 2$ and for any finite field $\Fq$ such that $A(q)>2d$ there exist families of codes $C$ 
such that both $C, C^{\bot}$ and the powers $C^{*d'}$ with $2\leq d'\leq d$ are simultaneously asymptotically good.
Interestingly, if we want to extend these results to other finite fields, the concatenation techniques of \cite{CCCX09} come to no avail, as opposed to the case of secret sharing schemes. Our results in the paper
show the existence of such asymptotically good families of codes for several small finite fields for which it was so far
not yet established. For instance, for $d=2$, we show that the result holds for any finite field $\Fq$, 
$q\geq 8$, except perhaps for $q=11, q=13$; in comparison, \cite{CC06} only showed this result in the case $A(q)>4$, which cannot hold when $q\leq 25$. 

Second, we consider bounds in the context of extension field
multiplication. Shparlinski, Tsfasman, and Vl\v{a}du\c{t} \cite{STV92} initiated study of asymptotics, finding upper bounds for the limits $m_q$, $M_q$ defined in that paper. We start by
noticing a gap in the proof of their main result: there is an implicit but unjustified assumption on
the possibilities of positive Ihara limits in combination with the absence of non-trivial 2-torsion. The
same gap exists in a more recent paper (2008) on the same subject by Ballet \cite{B08}. Therefore the
upper bounds stated for $m_q$ in those papers are not justified. On the other hand, Randriambololona recently proved in \cite{Randriam} that the bound for $m_q$ in \cite{STV92} can indeed be attained in the case $A(q)>5$. We examine the connection of this extension field multiplication problem to the solvability of a system of Riemann-Roch equations, and obtain
bounds that significantly improve the state of the art for some small fields by incorporating our limit and corresponding
tower. In addition, we also show how to improve the state of the art \cite{CCXY11} regarding the upper bounds for the other limit, $M_q$ over small finite fields $\Fq$.
Third, frameproof codes were introduced in the context of digital fingerprinting by Boneh and Shaw in \cite{BS98} although a slightly different definition, which we will be using, was proposed afterwards by Fiat and Tassa~\cite{FT99}, see also~\cite{B01}. The asymptotic properties of such codes have been studied in \cite{Randriam10, Randriam13, X02}. We show how to improve those bounds in some cases.

This paper is organized as follows. Our main contributions are captured in Definition~\ref{2.1} (the torsion limit), Theorem~\ref{2.2} (bounds for this limit),
Theorem~\ref{thm:system} (sufficient conditions for Riemann-Roch system solvability), Theorems~\ref{mth:newresulttor} and~\ref{mth:newresult} (claimed
arithmetic secret sharing schemes), Theorems~\ref{thm:mqbound} and~\ref{thm:upperbound} (improvements on multiplication complexity of finite field extensions) and Theorem~\ref{3.8} (improvements on asymptotical constructions for frameproof codes). After giving some preliminaries in Section~\ref{subsec:pre},
we introduce our torsion limit in Section~\ref{subsec:torsion} and show our
bounds. In Section~\ref{sec:rr} we introduce Riemann-Roch systems of
equations and show how these may be solved using the bounds from
Section~\ref{sec:tl}. We also include a short discussion on efficient randomized solving strategies. In Section~\ref{sec:lsss} we discuss how to obtain the claimed arithmetic secret sharing schemes
and linear codes with good duals and powers. In Section~\ref{sec:mult} we show how our torsion limit and Riemann-Roch system can be applied to
study the symmetric bilinear complexity of multiplication in finite fields. Finally in Section~\ref{sec:frame} we show our application to the asymptotical study of frameproof codes.

\section{Torsion Limits}\label{sec:tl}

\subsection{Preliminaries}\label{subsec:pre}
 For convenience of the reader, we start with some definitions and notations.

 For a prime power $q$, let $\FF_q$ be a finite field of $q$ elements. An {\em algebraic function field}  over
$\FF_q$ in one variable is a field extension $F \supset \FF_q$ such
that $F$ is a finite algebraic extension of  $\FF_q(x)$ for some
$x\in F$ that is transcendental over $\FF_q$. It is assumed that
$\FF_q$ is its full field of constants, i.e., the algebraic closure
of $\FF_q$ in $F$ is $\FF_q$ itself.

The following notations will be used throughout the rest of the
paper. \begin{itemize} \item $F/\F_q$--a function field with
full constant field $\F_q$; \item $g(F)$--the genus of $F$; \item
$N(F)$--the number of rational places of $F$; \item $\PP(F)$--the
set of places of $F$ (note that $\PP(F)$ is an infinite set);
\item $\PP^{(k)}(F)$--the set of places of degree $k$ of $F$ (note
that $\PP^{(k)}(F)$ is a finite set);
\item $N_i(F)$--the number of $\F_{q^i}$-rational places,
i.e., $N_i(F)=\sum_{j|i}j|\PP^{(j)}(F)|$ (note that $N(F)=N_1(F)$); \item
$\Div(F)$--the divisor group of $F$; \item $\Div^0(F)$--the
divisor group of degree $0$; \item $\Prin(F)$--the principal
divisor group of $F$; \item $\Cl(F)$--the divisor class group
$\Div(F)/\Prin(F)$ of $F$; \item $\Cl_0(F)=\mJ_F$--the degree zero
divisor class group $\Div^0(F)/\Prin(F)$ of $F$ (note that
$\Cl_0(F)$ is a finite group); \item $\mJ_F[r]$--the group of $r$-torsion points in $\mJ_F$. \item $h(F)=|\Cl_0(F)|$--the zero
divisor class number; \item $\mA_r(F)$--the set of effective
divisors of degree $r\ge 0$ (note that $\mA_r(F)$ is a finite
set); \item $A_r(F)$--the cardinality of $\mA_r(F)$; \item
$\Cl_r(F)$--the set $\{[D]:\; \deg(D)=r\}$, where $[D]$ stands
for the divisor class containing $D$. 
\end{itemize}
In case there is no confusion, we omit the function field $F$ in some of
the above notations. For instance, $A_r(F)$ is denoted by $A_r$ if
it is clear in the context.

For a divisor $G$ of $F$, we define the Riemann-Roch space by
\[\mL(G):=\{f\in F^*:\; {\rm div}(f)+G\ge 0\}\cup\{0\}.\]
Then $\mL(G)$ is a finite dimensional space over $\F_q$ and its
dimension $\ell(G)$ is determined by the Riemann-Roch theorem which
gives
\[\ell(G)=\deg(G)+1-g(F)+\ell(K-G),\]
where $K$ is a canonical divisor of degree $2g(F)-2$. Therefore, we
always have that $\ell(G)\ge \deg(G)+1-g(F)$ and the equality holds
if $\deg(G)\ge 2g(G)-1$.

The zeta function of $F$ is defined by the following power series
\[Z_F(t):=\Exp\left(\sum_{i=1}^{\infty}\frac{N_i(F)}it^i\right)=\sum_{i=0}^{\infty}A_i(F)t^i.\]
Then Weil showed that $Z_F(t)$ is in fact a rational function of the
form
\[Z_F(t)=\frac{L_F(t)}{(1-t)(1-qt)},\]
where $L_F(t)$ is a polynomial of degree $2g(F)$ in $\ZZ[t]$,
called {\it $L$-polynomial} of $F$. Furthermore, $L_F(0)=1$. If we
factorize $L_F(t)$ into a linear product
$\prod_{i=1}^{2g(F)}(w_it-1)$ in $\CC[t]$, then Weil showed that
$|w_i|=\sqrt{q}$ for all $1\le i\le 2g(F)$.

From the definition of zeta function, one obtains
\[N_m(F)=q^m+1-\sum_{i=1}^{2g(F)}w_i^m\]
for all $m\ge 1$. This gives the Hasse-Weil bound
\[N(F)=N_1(F)\le q+1+2g(F)\sqrt{q}.\]

Function fields $F$ with a large number $N(F)$ of rational points
have a bearing on problems in coding theory \cite{TVN07,St93} as well as, for instance, in low-discrepancy
sequences \cite{NX96} and several problems in cryptography
\cite{CC06, NWX07}. In particular, the following quantity is relevant:
\[N_q(g)=\max_{F}N(F),\]
where $F$ ranges over all function fields of genus $g$ over $\F_q$.

One can imagine that it is not easy at all to determine the exact
value $N_q(g)$ for an arbitrary pair $(q,g)$. The complete
solution to this problem has been found only for $g=0,1,2$
\cite{Serre85}. The reader may refer to \cite{Geer} for a table of
values of $N_q(g)$ for some small values of $q$ and $g$.

In order to study the asymptotic behavior of $N_q(g)$ when $q$ is fixed and
$g$ tends to $\infty$, we can define the following asymptotic
quantity
\[A(q):=\limsup_{g\rightarrow\infty}\frac{N_q(g)}g.\]
An upper bound on $A(q)$ was given by Vl\v{a}du\c{t} and Drinfeld
\cite{VD83}
\[A(q)\le \sqrt{q}-1.\]
For applications, we are more interested in finding lower bounds
on this asymptotic quantity. Ihara \cite{Ihara} first showed by
using modular curves that $A(q)\ge \sqrt{q}-1$ for any square power
$q$. This result determines the exact value $A(q)$ for all square
powers, i.e.,
\begin{equation}\label{eq2.1}A(q)=\sqrt{q}-1.\end{equation}

On the other hand, no single value of $A(q)$ is known if $q$ is a
non-square. However, some lower bounds have been obtained so far.
For instance, by using modular curves and explicit function fields,
Zink~\cite{Zink}, Bezerra-Garcia-Stichtenoth~\cite{BGS05} and Bassa-Garcia-Stichtenoth~\cite{BGS08} showed
that
\begin{equation}\label{eq2.2}A(q^3)\ge
\frac{2(q^2-1)}{q+2}.\end{equation}
Recently, Garcia-Stichtenoth-Bassa-Beelen \cite{GSBB12} produced an explicit tower of function fields over finite fields $\F_{p^{2m+1}}$ for any prime $p$ and integer $m\ge 1$ and showed that this tower gives
\[A(p^{2m+1})\ge \frac{2(p^{m+1}-1)}{p+1+\epsilon}\quad \mbox{with}\quad \epsilon=\frac{p-1}{p^m-1}.\]

 Serre made use of class field
theory to show that there is an absolute positive constant $c$ such
that
\[A(q)\ge c\cdot\log(q)\]
for every prime power $q$.

On the other direction, lower bounds
on $A(q)$ have already been obtained for small prime $q$ such as $q=2,3,5,7,11,13,\dots$
etc. For instance,  in \cite{XY07}, Xing and Yeo showed that
\[A(2)\ge 0.258.\]

For a family $\mF=\{F/\F_q\}$ of function fields with
$g(F)\rightarrow\infty$ such that
$\lim_{g(F)\rightarrow\infty}{N(F)}/{g(F)}$ exists, one can define
this limit to be the {\it Ihara} limit, denoted by $A(\mF)$. It is
clear that there exists a family $\mE=\{E/\F_q\}$ of function
fields such that $g(E)\rightarrow\infty$ and the Ihara limit
$A(\mE)$ is equal to $A(q)$.

\begin{remark} In general, we can define the Ihara limit for any
family $\mF=\{F/\F_q\}$ of function fields with
$g(F)\rightarrow\infty$ by
$\limsup_{g(F)\rightarrow\infty}{N(F)}/{g(F)}$. However, for
convenience of this paper, we define the Ihara limit only for
those families $\{E/\F_q\}$ whose limit
$\lim_{g(E)\rightarrow\infty}{N(E)}/{g(E)}$ exists.\end{remark}

\subsection{Torsion Limits}\label{subsec:torsion}
Due to some recent applications to arithmetic secret sharing and multiplications in finite field
extensions, we are interested in considering, in addition to the Ihara limit of a family of function fields,
a limit for the number of torsion
points of the zero divisor class groups of these function fields.

Let $F/\F_q$ be a function field. For a positive integer $r$
larger than $1$, we denote by $\mJ_F[r]$ the $r$-torsion point
group in $\mJ_F$, i.e.,
\[\mJ_F[r]:=\{[D]\in \mJ_F:\; r[D]=0\}.\]
The cardinality of $\mJ_F[r]$ is denoted by $J_F[r]$.

For each family $\mF=\{F/\FF_q\}$ of function fields with
$g(F)\rightarrow\infty$, we define the asymptotic limit
\[J_r(\mF):=\liminf_{F\in\mF}\frac{\log_q|\mJ_F[r]|}{g(F)}.\]

We need to define an asymptotic notion involving both $J_r(\mF)$ and the Ihara limit $A(\mF).$

\begin{definition}\label{2.1} For a prime power $q$, an integer $r>1$ and a real $a\leq A(q)$, let
$\mathfrak{F}$ be the set of families $\{\mF\}$ of function fields
over $\FF_q$
 such that the genus in each family tends to $\infty$ and the Ihara limit $A(\mF)\ge a$ for every $\mF\in\mathfrak{F}$.
Then the asymptotic quantity $J_r(q,a)$ is defined by
\[J_r(q, a)=\liminf_{\mF\in\mathfrak{F}}J_r(\mF).\]
\end{definition}

Thus, for a given family, our limit  $J_r(\mF)$ measures the $r$-torsion
against the genus. The corresponding constant $J_r(q,a)$ measures,
for a given Ihara limit $a$ and for given $r$, the ``least
possible $r$-torsion.'' Note that $A(q)$, Ihara's constant, is the
supremum of $A({\mF})$ taken over all asymptotically good ${\mF}$
over $\FF_q$. For some applications such as multiplication in extension
fields in Subsection 4.2, one may be interested in function fields with
many places of higher degree and small torsion limit. The above definition
could be modified by replacing the Ihara limit by the limit of number of places of higher degree against genus.

Now we are ready to state the main result of this
section.

\begin{theorem}\label{2.2}
Let $\FF_q$ be a finite field and let $r>1$ be a prime.
\begin{enumerate}
\item[{\rm (i)}] If $r\mid (q-1)$, then $J_r(q, A(q))\leq
\frac{2}{\log_r q}$. \item[{\rm (ii)}] If $r\nmid (q-1)$, then
$J_r(q, A(q))\leq \frac{1}{\log_r q}$. \item[{\rm (iii)}] If $q$ is
square and $r\mid q$, then $J_r(q, \sqrt{q}-1)\leq
\frac{1}{(\sqrt{q}+1)\log_r q}$.
\end{enumerate}
\end{theorem}

The {\em first} part of Theorem \ref{2.2}, as well as the second
part when, additionally, $r|q$, is proved directly  using a
theorem of Weil~\cite{Weil48,Mum70} on torsion in Abelian
varieties. For any non-zero integer $m$, this theorem,
which holds over {\em algebraically closed} fields $K$, says that the $m$-torsion point group $A[m]$ of the variety,
is isomorphic to $(\mathbb{Z}/m\mathbb{Z})^{2g}$ if $m$ is
co-prime to the characteristic $p$ of $K$; and $A[p]$ is
isomorphic to $(\mathbb{Z}/p\mathbb{Z})^{a}$ for a non-negative
integer $a\le g$, where $g$ is the dimension of $A$. See
also~\cite{Rosen}. Clearly, this implies upper bounds when the
field is not algebraically closed.  The second part, in the case
$r\nmid (q-1)$, can be proved by using the Weil
pairing for abelian varieties (see \cite{CCX11}). The most interesting part
is the bound in the {\em third part}, which is substantially
smaller (see Subsection~\ref{sub2.3} for the detailed proof). Note
that this last bound applies to families which attain the Drinfeld-Vl\v{a}du\c{t} bound.

By using a lifting idea, we are able to obtain an upper bound on
the size of the $r^t$-torsion point group of an abelian variety from its
$r$-torsion point group, and hence we can derive the following
result from  Theorem \ref{2.2} (see \cite{CCX11} for the detailed proof).

\begin{theorem}\label{2.3}
Let $\FF_q$ be a finite field of characteristic $p$.
\begin{enumerate}
\item[{\rm (i)}] If $m\ge 2$ is an integer, then $J_m(q,A(q))\le
\log_q({dm})$, where $d=\gcd(m,q-1)$. \item[{\rm (ii)}] Write $m$
into $p^{\ell}m'$ for some $\ell\ge 0$ and an positive integer $m'$
co-prime to $p$. If $q$ is a square, then $J_m(q,\sqrt{q}-1)\le
\frac{\ell}{\sqrt{q}+1}\log_q(p)+\log_q(cm')$, where
$c=\gcd(m',q-1)$.
\end{enumerate}
\end{theorem}

At the time when an earlier version of this paper was being prepared, Randriambololona independently introduced in~\cite{Randriam10}
the limit $J_r(q,A(q))$, in the context of an application to the construction of frameproof codes. Moreover he stated
the bounds in the first part of Theorem~\ref{2.2}, the second part of Theorem~\ref{2.2} when $r|q$ and the first part of
Theorem~\ref{2.3} when $m$ is a power of $p$.

Like the Ihara-constant $A(q)$, it could be extremely difficult to
determine the exact value of $J_r(q,a)$ for given $a$ and $q$, and
we would like to leave this as an open problem. Also, in the
context of solving general Riemann-Roch systems (see Section 3) it
makes sense to extend the definition of the limit above to the
case of $r$-torsion for a finite set of positive integers $r$ {\em
simultaneously}.

Another particular interesting case is $q=2$.  The
following result gives a bound on the $2$-torsion limit for the
family of function fields given in \cite{XY07}.

\begin{theorem}\label{2.4}
The family $\mF$ of function fields over $\F_2$ with the Ihara's
limit $97/376$ given in \cite{XY07} has $2$-torsion limit
$J_2({\mathcal F})$ at most $216/376$.
\end{theorem}

The proof of Theorem \ref{2.4} will be  given in Subsection \ref{sub2.3}. Note
that the bound in Theorem \ref{2.2} gives only $J_2({\mathcal
F})\le 1$.

Finally, one can show existence of certain function field families that
is essential for our applications of Sections~\ref{sec:lsss} and~\ref{sec:mult}.

\begin{theorem}\label{2.4a}
For every prime power $q\ge 8$ except perhaps for $q=11$ or $13$, there exists a family
${\mathcal F}$ of function fields over $\F_q$ such that the Ihara
limit $A({\mathcal F})$ exists and it satisfies ${A(\mathcal F})>
1+J_2({\mathcal F}).$
\end{theorem}
Again we refer to  \cite{CCX11} for the detailed proof of {Theorem} \ref{2.4a}.

\subsection{Proof of Theorems \ref{2.2}(iii) and
\ref{2.4}}\label{sub2.3}

Let $\FF_q$ be a finite field. Write $p$ for its characteristic.
For a function field $F$ over $\FF_q$, denote by $\gamma(F)$ the $\FF_p$-dimension of $\mJ_F[p]$, i.e., $\log_p |\mJ_F[p]|$.\footnote{Note that in the definition of $\gamma(F)$ the logarithm of $\mJ_F[p]$ is taken in base $p$ as opposed to the definition of $J_p(\mF)$, where it is taken in base $q$.} 
Now, consider the constant field extension $\overline{F}=F\cdot\overline{\FF}_q$ where $\overline{\FF}_q$ denotes an algebraic closure of $\Fq$. Then the {\it Hasse-Witt} invariant $i_F$ of $F$ is defined to be the $\FF_p$-dimension of $\mJ_{\overline{F}}[p].$
It is clear that $\mJ_F[p]$ is an $\FF_p$-subspace of $\mJ_{\overline{F}}[p]$, and hence $i_F\ge \gamma(F)$.

Note that, for any family $\mF$ of function fields $F$ over $\Fq$ with $g(F)\rightarrow\infty$, 
$$J_p(\mF)=\liminf_{F\in\mF}\frac{\gamma(F)}{g(F)\log_p q}\leq \liminf_{F\in\mF}\frac{i_F}{g(F)\log_p q}.$$ 

Before proving Theorems \ref{2.2}(iii) and \ref{2.4},
we need to introduce the Deuring-Shafarevich theorem.
\begin{theorem}[Deuring-Shafarevich  (see e.g.~\cite{HKT08})]
Let $E/F$ be a Galois extension of function fields over an algebraically closed field $k$ of characteristic $p$.
 Suppose that the Galois group of the
extension is a $p$-group. Then
$$\gamma(E)-1=[E:F](\gamma(F)-1)+\sum_{P\in\Proj(F)}\sum_{\substack{Q\in\Proj(E) \\\  Q|P}}(e(Q|P)-1).$$
\end{theorem}

From this theorem, we can obtain the following corollary for function fields over \emph{finite} fields.

\begin{corollary}\label{cor:deuringfinite}
 Let $E/F$ be a Galois extension of function fields over a finite field $\Fq$ of characteristic $p$. Suppose that the Galois group of the
extension is a $p$-group. Then
$$i_E-1=[E:F](i_F-1)+\sum_{P\in\Proj(F)}\sum_{\substack{Q\in\Proj(E) \\\  Q|P}}(e(Q|P)-1)\deg Q.$$
\end{corollary}
\begin{proof}
 Let $\overline{E}=E\cdot\overline{\FF}_q$, $\overline{F}=F\cdot\overline{\FF}_q$ where $\overline{\FF}_q$ denotes an algebraic closure of $\Fq$.
 By elementary algebra arguments we can see that since $E/F$ is Galois and both $E$ and $F$ have the same full constant field $\Fq$,
then $\overline{E}/\overline{F}$ is also Galois and the Galois groups of both extensions are the same.

 We can therefore apply the Deuring-Shafarevich Theorem to $\overline{E}$ and $\overline{F}$, thereby obtaining:
$$\gamma(\overline{E})-1=[\overline{E}:\overline{F}](\gamma(\overline{F})-1)+\sum_{P'\in\Proj(\overline{F})}\sum_{\substack{Q'\in\Proj(\overline{E}) \\\  Q'|P'}}(e(Q'|P')-1).$$

Note that $\gamma(\overline{E})=i_E$, $\gamma(\overline{F})=i_F$ and $[\overline{E}:\overline{F}]=[E:F]$, so all we are left to do is to analyse the last term.

Given a place $P\in\Proj(F)$ of degree $k$, and a place $Q\in\Proj(E)$ of degree $m$ such that $Q|P$, there are exactly $k$ places $P'_1,\dots,P'_k\in\Proj(\overline{F})$ lying over $P$ and $m$ places $Q'_1,\dots,Q'_m\in\Proj(\overline{E})$ lying over $Q$. Each of the places $Q'_j$ lies above some $P'_i$. Moreover, all places of $\overline{E}$ lying above a place $P'_i\in\Proj(\overline{F})$ are among the $Q'_j$.
It is well known that all places in $\overline{F}$ and $\overline{E}$ have degree $1$.
Given $P'$ in $\{P'_1,\dots,P'_k\}$ and $Q'$ in $\{Q'_1,\dots,Q'_m\}$, we have $e(P'|P)=1$ and $e(Q'|Q)=1$. Consequently if $Q'$ lies above $P'$, we deduce $e(Q'|P')=e(Q|P)$ since $e(Q'|P')e(P'|P)=e(Q'|P)=e(Q'|Q)e(Q|P)$.

Thus
$$\sum_{P'\in\Proj(\overline{F})}\sum_{\substack{Q'\in\Proj(\overline{E}) \\\  Q'|P'}}(e(Q'|P')-1)=\\$$
$$\sum_{P\in\Proj(F)}\sum_{\substack{Q\in\Proj(E) \\\  Q|P}}(e(Q|P)-1)\deg Q. $$

\end{proof}

Now we are ready to prove Theorem~\ref{2.2}(iii).

\begin{proof}[Proof of Theorem \ref{2.2}(iii)]
Assume $q$ is an even power of $p$. Consider the tower ${\mathcal F}=(F^{(0)}
\subset F^{(1)} \subset \cdots)$ over $\FF_q$ introduced in
\cite{GSINV} by Garcia and Stichtenoth, recursively defined by $F^{(0)}=\F_q(x_0)$ and
$F^{(n+1)}=F^{(n)}(x_{n+1})$, where
$x_n^{\sqrt{q}-1}x_{n+1}^{\sqrt{q}}+x_{n+1}=x_n^{\sqrt{q}}$. 

Assuming for the moment that Theorems~\ref{thm:GSINV} and~\ref{2.5} stated below hold, the rest of the argument
follows immediately: indeed,
 $$\liminf_{n \to \infty}
\frac {\gamma(F^{(n)})}{g(F^{(n)})}\leq \lim_{n \to\infty} \frac{i_{F^{(n)}}}{g(F^{(n)})}=\frac{1}{\sqrt{q}+1}.$$
where the equality follows from part 3 of Theorem~\ref{thm:GSINV} and Theorem~\ref{2.5}.

Therefore
$$J_{p}(\mF)= \liminf_{n \to \infty}
\frac {\gamma(F^{(n)})}{g(F^{(n)}) \log_p q}= \frac{1}{ (\sqrt{q}+1)\log_p q}.$$

On the other hand, $A({\mathcal F})=\sqrt{q}-1$ by part 1 of Theorem~\ref{thm:GSINV}. Therefore
$$J_p(q,\sqrt{q}-1)= \frac{1}{ (\sqrt{q}+1)\log_p q}.$$ It only remains to show Theorems~\ref{thm:GSINV} and~\ref{2.5}.

\begin{theorem}\label{thm:GSINV}
\begin{enumerate}
\item
The tower $\mathcal F$ attains the Drinfeld-Vl\v{a}du\c{t} bound, i.e.,
its limit $A({\mathcal F})$ is given by $$A({\mathcal
F}):=\lim_{n\to \infty}\frac{N(F^{(n)})}{g(F^{(n)})}=\sqrt{q}-1.$$
\item Every place $P \in \Proj(F^{(n-1)})$ is either {\em unramified}, i.e. for every place $Q \in \Proj(F^{(n)})$ such that
$Q|P$ we have $e(Q|P)=1$, where $e(Q|P)$ denotes the ramification index, or {\em totally ramified}, i.e., there exists a unique $Q \in \Proj(F^{(n)})$ such that
$Q|P$, and the ramification index $e(Q|P)$ equals $[F^{(n)}:F^{(n-1)}]=\sqrt{q}$. In the latter case, it always holds that $\deg P=\deg Q$.
Moreover for every $P \in \Proj(F^{(n-1)})$, $Q \in \Proj(F^{(n)})$ such that
$Q|P$ we have $$d(Q|P)=(\sqrt{q}+2)(e(Q|P)-1),$$ where $d(Q|P)$ denotes the different exponent.
\item The genus $g(F^{(n)})$ of the function field $F^{(n)}$ is given by

$$g(F^{(n)})=\left\{
   \begin{array}{ll}g_1(q,n) & \makebox{ if $n \equiv 0 \pmod{2}$,}\\
      \\
   g_2(q,n) & \makebox{ if $n \equiv 1 \pmod{2}$.}\\
   \end{array}
\right.$$

where $$g_1(q,n):={q}^{\frac{n+1}2}+{q}^{\frac n2}-{q}^{\frac{n+2}4}-2{q}^{\frac n4}+1,$$
$$g_2(q,n):={q}^{\frac{n+1}2}+{q}^{\frac n2}-\frac12{q}^{\frac{n+3}4}-\frac32{q}^{\frac{n+1}4}-{q}^{\frac{n-1}4}+1.$$

\end{enumerate}
\end{theorem}

\begin{proof}
 See~\cite{GSINV}.
\end{proof}

\begin{theorem}\label{2.5}
The Hasse-Witt invariant of the function field $F^{(n)}$ is given by
$$i_{F^{(n)}}=\left\{ \begin{array}{ll}
(q^{n/4}-1)^2 & \makebox{ if $n \equiv 0 \pmod{2}$,}\\
\\
(q^{(n-1)/4}-1)(q^{(n+1)/4}-1) & \makebox{ if $n \equiv 1 \pmod{2}$.}\\
\end{array}\right.$$

In particular $$\liminf_{n \to \infty}
\frac {\gamma(F^{(n)})}{g(F^{(n)})}\leq \lim_{n \to\infty} \frac{i_{F^{(n)}}}{g(F^{(n)})}=\frac{1}{\sqrt{q}+1}.$$
\end{theorem}

\begin{proof}

Fix some $n\geq 1$ and for the sake of notation let $E:=F^{(n)}$, $F:=F^{(n-1)}$. Consider the extension $E/F$. This is an
Artin-Schreier extension, hence its Galois group is a
$p$-group. By the theorem of Riemann-Hurwitz (see e.g.~\cite{St93})
and part 2) of Theorem~\ref{thm:GSINV} above,

\begin{equation}\label{eq:one}
2\cdot g(E)-2=\sqrt{q}\cdot (2g(F)-2)+(\sqrt{q}+2)\cdot\sum_{P\in\Proj(F)}\sum_{\substack{Q\in\Proj(E) \\\ Q|P}}(e(Q|P)-1)\deg Q.
\end{equation}

By Corollary~\ref{cor:deuringfinite}

\begin{equation}\label{eq:two}
i_E-1=\sqrt{q}\cdot
(i_F-1)+\sum_{P\in\Proj(F)}\sum_{\substack{Q\in\Proj(E)
\\\  Q|P}}(e(Q|P)-1)\deg Q.
\end{equation}

Combining equations (\ref{eq:one}) and (\ref{eq:two}), we find

$$i_E=\sqrt{q} \cdot i_F+\dfrac{2\cdot g(E)-2\sqrt{q}
\cdot g(F)-\sqrt{q}^2+\sqrt{q}}{\sqrt{q}+2}$$

This, of course, holds for any $n\geq 1$, $E:=F^{(n)}$, $F:=F^{(n-1)}$. Using the fact that $i_{F^{(0)}}=0$ and applying induction,
the result follows.
\end{proof}

This concludes the proof of Theorem \ref{2.2}(iii).
\end{proof}

We can use the same kind of argument applied to a different tower to prove Theorem~\ref{2.4}:

\begin{proof}[Proof of Theorem~\ref{2.4}:]
In \cite{XY07}, Xing and Yeo gave an example of a tower ${\mathcal F}=(F^{(0)}
\subset F^{(1)} \subset \cdots)$ of function fields over $\F_2$ with the Ihara limit
$97/ 376 = 0.257979\dots$. 
Using cyclotomic function fields, they
constructed a function field $F = F^{(0)}$ over $\F_2$ of genus $377$,
which admits an infinite $(2; S)$-Hilbert class field tower for a
set $S \subset \PP_F$ of places of $F$, such that $S'
=\PP_F\setminus S$ consists of $97$ rational places of $F$. At each
step $F^{(i+1)}/F^{(i)}$, it is unramified. Hence, to compute the Hasse-Witt invariant of
$F^{(i)}$, it is sufficient to compute the Hasse-Weil invariant of $F^{(0)}$ by using
the formula of Deuring-Shafarevich.

To do so, we briefly recall the construction of the function field
$F$. For more details, the reader may refer to \cite{XY07}. Let
$k=\F_2(x)$ be the rational function field over $\F_2$. Let $M =
(x^4 +x^3 +x^2 +x+1)^2 \in \F_2[x]$ and let $N := x^4$. Denote by
$k_M$ (resp. $k_N$) the cyclotomic function field over $k$ with
modulus $M$ (resp. modulus $N$). Let $K$ be the subfield of $k_M$
fixed by the cyclic subgroup $< x >$ of Gal$(k_M/k) = (\F_2[x]/ M
)^*$ and let $L$ be the subfield of $k_N$ that is fixed by the
cyclic subgroup $< (x + 1)^2>$ of Gal$(k_N/k) = (\F_2[x]/N )^*$.
We have $[K : k] = 24$ and $[L : k] = 4$. Define $F := KL$, the
composite of the fields $K$ and $L$. The only ramified place in
$K/k$ is the place corresponding to the irreducible polynomial
$x^4 + x^3 + x^2 + x + 1$. It is totally ramified with different
exponent $44$. In the extension $L/k$ the only ramified place is
the zero of $x$. It is totally ramified with different exponent
$10$.

 From the ramification
in $K/k$ and $L/k$, it follows that $K$ and $L$ are linearly
disjoint over $k$. We have $[F : k] = 2^5\times 3$. The fixed field
of the $2$-Sylow subgroup of Gal$(F/k)$ is generated over $k$ by an
element $w$, whose irreducible polynomial over $k$ is given by 
$$T^3+ (x^4 + x^3 + x^2 + x + 1)  T^2 + (x^5 + 1)  T + (x^4 + x^3 + x^2 +
x + 1) \in k[T].$$
 Let $F' = k(w)$. We have $k\subset F' \subset K$.
The only ramified place in $F'/k$ is the place corresponding to the
irreducible polynomial $x^4+x^3+x^2+x+1$. It is tamely ramified with
ramification index $3$. Hence the genus of $F'$ is $2$. Next by
computing the Hasse-Witt invariant of $F$ we know that in the degree
$32$ extension $F/F'$ the only ramified places are the places lying
over the places of $k$ associated to the irreducible polynomials $x$
and $x^4 + x^3 + x^2 + x + 1$. The corresponding ramification
indices are $4$ and $8$, respectively. So we have

\[i_F - 1 = 32  (2 - 1) + 4 \times 4 \times (8 -1) + 3 \times 8 \times (4 - 1) =
216.\]
 For the $(2; S)$-Hilbert class field tower of $F = F^{(0)}$, we hence
have \[g(F^{(n)}) - 1 = [F^{(n)} : F^{(0)}] (g(F^{(0)}) - 1) = 376  [F^{(n)} : F^{(0)}]\]
and \[i_{F^{(n)}} - 1 = [F^{(n)} : F^{(0)}]  (i_{F^{(0)}} - 1) = 216 [F^{(n)} :
F^{(0)}].\]
 Therefore,\[ \lim_{n\rightarrow\infty}\frac{i_{F^{(n)}}}{g(F^{(n)})} = \frac{216}{376}
 =0.574468\dots.\]
\end{proof}
The Deuring-Shafarevich theorem has been used in \cite{BB10} to analyze the $p$-rank
 of the function fields for other optimal towers over $\Fq$, for $q$ square. However,
 the resulting bounds for the torsion limits for those towers are worse (for our 
applications) than that of the first Garcia-Stichtenoth tower, obtained above.

\section{Riemann-Roch Systems of Equations}\label{sec:rr}

Let $\FF_q$ be a finite field and let $F$ be an algebraic function
field over $\FF_q$.

\begin{definition}
Let $u\in \ZZ_{>0}$ and let $Y_i\in \mbox{Cl}(F)$,
$m_i\in\ZZ\setminus{\{0\}}$ for $i=1,\dots,u$. The
\emph{Riemann-Roch system of equations} in the indeterminate $X$  is
the system  $\{\ell(m_iX+Y_i)=0\}_{i=1}^u$ determined by these data.
A solution is some $[G]\in \mbox{Cl}(F)$ which satisfies all
equations when substituted for $X$.
\end{definition}

While Riemann-Roch systems have been (implicitely) used before in the construction of codes with good asymptotic properties, for instance in~\cite{V87,X01,CC02,X05,Xu05,M07,NO07}, they were of a less general type.
Namely, $m_i=\pm 1$ for all $i$. As we shall see soon, dealing with the more general case where $m_i\neq \pm 1$ leads us to consider $m_i$-torsion in the class group.

One observation about the systems is that $X$ is a solution of the equation
$\ell(m_iX+Y_i)=0$ as long as $\deg(m_iX+Y_i)<0$ since we have
$\ell(m_iX+Y_i)=0$ in this case. This suggests that, if we want to prove the existence of solutions of certain fixed degree, we should only
consider those equations $\ell(m_iX+Y_i)=0$ in the Riemann-Roch
system with $\deg(m_iX+Y_i)\ge0$.

The following theorem shows that a solution of degree $d$ exists
if a certain numerical condition is satisfied that involves the
class number, the number $A_{r_i}$ of effective divisors of degree
$r_i$ and the cardinality of the $m_i$-torsion subgroups of the
degree-zero divisor class group, where the $m_i$ are determined by
the system and the $r_i$ are determined by $d$ and the $m_i$.

\begin{theorem}\label{thm:system}
Consider the Riemann-Roch system of equations
$$\{\ell(m_iX+Y_i)=0\}_{i=1}^u.$$ Let $d_i=\deg Y_i$ for
$i=1,\dots,u$. Write $h:=h(F)$ the class number. Denote by
$A_r$ the number of effective divisors of degree~$r$ in
$\Div(F)$ for $r\ge 0$, and $0$ for $r<0$. Let $s\in \ZZ$
and define $r_i=m_is+d_i$ for $i=1, \ldots, u$.
If $$h>\sum_{i=1}^u A_{r_i}\cdot|\mJ_F[{m_i}]|,$$ then the Riemann-Roch system has a
solution $[G]\in \mbox{Cl}_s(\FF)$.
\end{theorem}
We refer to \cite{CCX11} for the detailed proof of Theorem \ref{thm:system}.

\begin{remark}  (``Solving by taking any divisor $X$  of large enough degree'')\label{rem:largedegree}
\begin{itemize}
\item[(i)] If $r_i<0$ for all $i=1,\dots,u$, then the inequality
in Theorem \ref{thm:system} is automatically satisfied and hence the
Riemann-Roch system always has a solution.

\item[(ii)] In many
scenarios in algebraic geometry codes, one can simply argue for a
solution of the Riemann-Roch system by assuming that $r_i<0$ for all
$i=1,\dots,u$.

\item[(iii)] For instance, in~\cite{CC06}, it was also simply assumed $r_i<0$ to obtain strongly multiplicative linear secret sharing schemes. But this does not always give the best results.
 In particular, in Section~\ref{sec:lsss}, we will show how we can employ Theorem
\ref{thm:system} to get improvements, {\em especially for small
finite fields}.
\end{itemize}

\end{remark}

It will often be more convenient to write systems as defined over
$\mbox{Div}(F)$ rather than $\mbox{Cl}(F)$.

The condition in Theorem~\ref{thm:system} involves
the number of positive divisors of certain degrees and the class number.
The following bound will be useful in the applications. The proof is based on careful
manipulations with the zeta function of $\F$.

\begin{proposition}\label{propo:Arh}
 Let $F$ be an algebraic function field over $\FF_q$.
Write $g$  for the genus $g(F)$ and $h$  for the class number
$h(F)$. For $r\in \ZZ_{\geq 0}$, write $A_r$ for the number of
effective divisors of degree~$r$ in $\Div(\FF)$. Suppose
$g\geq 1$. Then, for any integer $r$ with $0\leq r \leq g-1$,
$$\frac{A_r}{h}\leq \frac{g}{q^{g-r-1}(\sqrt{q}-1)^2}.$$
\end{proposition}

A very brief proof of Proposition \ref{propo:Arh} was given in \cite{CCX11}. Here we give a detailed proof.

\begin{proof}
For $i\geq 2g-1$ the value of $A_i$ is known as a function of
$q,g,h,i$ (see Lemma~{5.1.4} and Corollary~{5.1.11} in~\cite{St93}).
This has been exploited in
 Lemma~3~{(ii)} from~\cite{NX96}, to show that
$$\sum_{i=0}^{g-2} A_it^i + \sum_{i=0}^{g-1}q^{g-1-i}A_it^{2g-2-i}=\frac{L(t)-ht^g}{(1-t)(1-qt)}$$
by  manipulations of power series, where $L(t)$ is the $L$-polynomial in the zeta function of $F$.

The claim will be derived from a relation
that is obtained by taking the limit as $t$ tends to $1/q$ on both
sides of the equation above, where  l'H\^{o}pital's Rule  is applied
on the RHS, then finding an expression for $L'(1/q)$ (the
``left-over term''), and substituting that back in.

Taking this limit,
$$\sum_{i=0}^{g-2}\frac{A_i}{q^i}+\sum_{i=0}^{g-1}\frac{A_i}{q^{g-1}}=\lim_{t\rightarrow 1/q}\frac{L(t)-ht^g}{(1-t)(1-qt)},$$
and applying l'H\^{o}pital's rule ($(f(t))'|_{t=a}$ denotes the
derivative of $f$ evaluated at $t=a$), it follows that

$$\frac{(L(t)-ht^g)'|_{t=1/q}}{((1-t)(1-qt))'|_{t=1/q}}
=\frac{L'(1/q)-gh/q^{g-1}}{-q(1-1/q)}=$$
$$
=\frac{gh-q^{g-1}L'(1/q)}{(q-1)q^{g-1}}.
$$

The term $L'(1/q)$ can be evaluated as follows. By differentiation,
$$L'(t)=\sum_{i=1}^{2g} L(t)\cdot \frac{-\omega_i}{1-\omega_i t},$$
and
hence,$$L'\left(\frac{1}{qt}\right)=L\left(\frac{1}{qt}\right)\cdot
\sum_{i=1}^{2g}(qt)\cdot \frac{-\omega_i}{qt-\omega_i}.$$ Evaluation
of  $L(1/q)$ is straightforward by combining the Functional Equation
for $L$-polynomials and the fact that $L(1)=h$ (see ~\cite{St93}).
 Namely,
 $$L\left(\frac1q\right)=q^{g}\left(\frac1q\right)^{2g}L(1)=\frac h{q^g}.$$
 Therefore,
$$L'\left(\frac1q\right)=\frac{h}{q^{g-1}}\cdot \sum_{i=1}^{2g}\frac{-\omega_i}{q-\omega_i}.$$

Substituting the expression for $L'(1/q)$ back in, it follows that
$$\sum_{i=0}^{g-2}\frac{A_i}{q^i}+\sum_{i=0}^{g-1}\frac{A_i}{q^{g-1}}=
\frac{h}{q^{g-1}(q-1) }\cdot \left(g+\sum_{i=0}^{2g}
\frac{\omega_i}{q-\omega_i}\right).$$

Note that, trivially, by writing it appropriately as a fraction of
the other expressions in the equation, the expression between
brackets on the right-most side must be a positive number. Using
this and the fact  $|\omega_i|=\sqrt{q}$ for $i=1,\ldots, 2g$, it
holds, for $0\leq r\leq g-1$, that {\small
$$
\frac{A_r}{q^r}\leq
\sum_{i=0}^{g-2}\frac{A_i}{q^i}+\sum_{i=0}^{g-1}\frac{A_i}{q^{g-1}}=
\frac{h}{q^{g-1}(q-1)}\cdot \left|g+\sum_{i=0}^{2g}
\frac{\omega_i}{q-\omega_i}\right|$$
$$\leq \frac{h}{q^{g-1}(q-1) }\cdot
\left(g+\sum_{i=0}^{2g} \frac{|\omega_i|}{q-|\omega_i|}\right)=
\frac{gh}{q^{g-1}(q-1)}\cdot \left(1+\frac{2}{\sqrt{q}-1}\right)$$
$$
=\frac{gh}{q^{g-1}(q-1)}\cdot
\left(\frac{\sqrt{q}+1}{\sqrt{q}-1}\right) =\frac{gh}{q^{g-1}\cdot
(\sqrt{q}-1)^2}.
$$}
and  the claimed result follows.
\end{proof}

\begin{remark}[Efficient randomized solving strategy]
Except in the cases where Remark~\ref{rem:largedegree} applies, we are not aware of any efficient deterministic strategy to solve efficiently the Riemann Roch 
systems of equations appearing in our applications.
However, in many circumstances, there is an efficient \emph{randomized} strategy that produces a divisor that is a solution with \emph{high probability}.
Namely, if the number of non-solutions (which is bounded by the right hand side of the inequality in Theorem~\ref{thm:system}) is negligible
as a function of the class number, then a uniformly random element from $\mbox{Cl}_s(\FF)$
will be a solution with overwhelming probability.
Assuming both means for efficient sampling in $\mbox{Cl}_s(\FF)$ according to a distribution sufficiently close to uniform and for efficient construction of generator matrices of the algebraic geometric codes associated to the sampled divisors, there exist efficient probabilistic constructions of the asymptotically good families of codes in our applications.
As for the sampling issue, we note that it can be done under mild conditions (see \cite[Section 5.3.2.]{Boe} or \cite[Theorem 5]{Hess12}). As for the 
construction of generator matrices, not much is known in full generality but results for the construction of bases of Riemann-Roch spaces of
general divisors can be found in~\cite{Hess12}. Generator matrices for algebraic geometric codes on function fields of a tower of Garcia-Stichtenoth were explicitely constructed in~\cite{SAK01} although only one-point divisors are considered.

\end{remark}

\section{Application 1: Arithmetic Secret Sharing}\label{sec:lsss}

Our first application concerns the asymptotic study of \emph{arithmetic secret sharing schemes},
which was first considered in~\cite{CDM00,CC06} in the context of secure multi-party computation.
Since then, the asymptotical results from~\cite{CC06} have had important and surprising applications in {\em two-party} cryptography as well~\cite{IKOS07,IPS08,HIKN08,IKOS09,DIK10,IKOPSW11}. For a more detailed discussion of the motivation, results and applications, please refer to~\cite{CCX11}.

One motivation of this section is to show an important application of our torsion limits and Riemann-Roch systems introduced in the previous sections. As the proofs  of most results in this section can be found in \cite{CCX11}, we state our results without detailed proof.

We first define arithmetic secret sharing schemes and  then show how our torsion limits help to improve prior results significantly.

Let $k,n$ be integers with $k,n\geq 1$. Consider the $\Fq$-vector space $\Fq^k\times\Fq^n$, where $\Fq$ is an arbitrary finite field.

\begin{definition} 
The $\Fq$-vector space morphism
$$\pi_0:\Fq^k\times\Fq^n\rightarrow \Fq^k$$
is defined by the projection  $$(s_1,\dots,s_k,c_1,\dots,c_n)\mapsto(s_1,\dots,s_k).$$

For each $i\in\{1,\dots,n\}$, the $\Fq$-vector space morphism   $$\pi_i:\Fq^k\times\Fq^n\rightarrow \Fq$$ is defined by the projection $$(s_1,\dots,s_k,c_1,\dots,c_n)\mapsto c_i.$$
For $\emptyset\neq A\subset\{1,\dots,n\}$, the $\Fq$-vector space morphism $$\pi_A:\Fq^k\times\Fq^n\rightarrow \Fq^{|A|}$$ is defined by the projection  $$(s_1,\dots,s_k,c_1,\dots,c_n)\mapsto (c_i)_{i\in A}.$$
For ${\bf v}\in \Fq^k\times\Fq^n$, it is sometimes convenient to denote $\pi_0({\bf v})\in \Fq^k$ by ${\bf v}_0$  and $\pi_A({\bf v})\in \Fq^{|A|}$ by ${\bf v}_A$. We write ${\mathcal I}^*=\{1, \ldots, n\}$.
It is also sometimes convenient to refer to ${\bf v}_0$ as the {\em secret-component} of ${\bf v}$ and to
${\bf v}_{{\mathcal I}^*}$ as its {\em shares-component}.
\end{definition}

\begin{definition} 
 An {\em $n$-code for $\Fq^k$ ({\em over} $\Fq$)} is an $\Fq$-vector space $C\subset\Fq^k\times\Fq^n$ such that
\begin{itemize}
\item [{\rm (i)}] $\pi_0(C)=\Fq^k$
\item [{\rm (ii)}] $(\mathrm{Ker\ }\pi_{\mI^{*}})\cap C \ \subset \  (\mathrm{Ker\ }\pi_0)\cap C$.
\end{itemize}
For ${\bf c}\in C$, ${\bf c}_0\in \Fq^k$ is the {\em secret} and ${\bf c}_{{\mathcal I}^*}\in \Fq^n$ the {\em shares}.
\end{definition}

The first condition means that, in $C$, the  secret can take any value in $\Fq^k$.
More precisely, for a uniformly random vector ${\bf c}\in C$,
the secret ${\bf c}_0$ is uniformly random in $\Fq^k$. This follows from the fact that the projection $(\pi_0)_{|C}$ is regular (since it is a surjective $\Fq$-vector space morphism).\\
The second condition means that the shares uniquely determine the secret.
Indeed, the shares do not always determine the secret uniquely if and only if there are ${\bf c}, {\bf c}'\in C$ such that their shares coincide but not their secrets.
Therefore, by linearity, the shares determine the secret uniquely if and only if the shares being zero implies the secret being zero. Moreover these two conditions imply that $k\leq n$.\\
Note that an $n$-code with the stronger condition $(\mathrm{Ker\ }\pi_{\mI^{*}})\cap C= (\mathrm{Ker\ }\pi_0)\cap C$ is a $k$-dimensional error correcting code of length $n$.

\begin{definition}[$r$-reconstructing] 
An $n$-code $C$ for $\Fq^k$ is {\em $r$-reconstructing} \ ($1\leq r\leq n$) if
$$
(\mathrm{Ker\ }\pi_A)\cap C \subset (\mathrm{Ker\ }\pi_0)\cap C
$$
for each
$A\subset \mI^{*}$ with $|A|=r$.
\end{definition}

In other words, $r$-reconstructing means that any $r$ shares
uniquely determine the secret. Note that an $n$-code is $n$-reconstructing by definition.

\begin{definition}[$t$-Disconnected] 
An $n$-code $C$ for $\Fq^k$ is {\em $t$-disconnected} if $t=0$ or else if $1\leq t<n$,
the projection
$$
\pi_{0,A}: C \longrightarrow \Fq^k \times \pi_A(C)$$
$$
\bc \mapsto (\pi_0(\bc), \pi_A(\bc))
$$
is surjective for each $A\subset\mI^{*}$ with $|A|=t$.\\
If, additionally, $\pi_A(C)=\Fq^t$,
we say $C$ is {\em $t$-uniform}.
\end{definition}

If $t>0$, then $t$-disconnectedness means the following. Let $A\subset {\mathcal I}^*$ with $|A|=t$.
Then, for  uniformly randomly  ${\bf c}\in C$,
the secret ${\bf c}_0$ is independently distributed from the $t$  shares ${\bf c}_A$. Indeed, for the same reason that the secret ${\bf c}_0$ is uniformly random in $\Fq^k$, it holds that  $({\bf c}_0, {\bf c}_A)$ is uniformly random in $\Fq^k\times \pi_A(C)$. Since the uniform distribution on the Cartesian-product of two finite sets corresponds to the uniform distribution on one set, and independently, the uniform distribution on the other, the claim follows. Uniformity means that, in addition, ${\bf c}_A$ is uniformly random in $\Fq^t$.

\begin{definition} [Powers of an $n$-Code] \label{def:powers} Let $m\in \ZZ_{>0}$.
For ${\bf x}, {\bf x}'\in \Fq^m$, their {\em product} ${\bf x}*{\bf x}'\in \Fq^m$ is defined as
$(x_1x'_1, \ldots, x_mx'_m)$.\\
 Let $d$ be a positive integer. If $C$ is an $n$-code for $\Fq^k$, then
$C^{*d}\subset \Fq^k\times \Fq^n$ is the $\Fq$-linear subspace generated by all terms of the form
${\bf c}^{(1)}*\ldots *{\bf c}^{(d)}$ with ${\bf c}^{(1)}, \ldots,  {\bf c}^{(d)}\in C$.
For $d=2$, we use the abbreviation $\widehat{C}:=C^{*2}$.
Powers of linear codes (instead of $n$-codes) are defined analogously and will be useful later.
\end{definition}

\begin{remark}[Powering Need Not Preserve $n$-Code]
Suppose $C\subset \Fq^k\times \Fq^n$ is an $n$-code for $\Fq^k$. It follows immediately that the secret-component in $C^{*d}$ takes any value in $\Fq^k$. {\emph However},
the shares-component in $C^{*d}$ {\em need not} determine the secret-component uniquely.
Thus, {\em $C^{*d}$ need not be an $n$-code for $\Fq^k$}.
\end{remark}

\begin{definition}[Arithmetic secret sharing scheme] 
An {\em $(n, t, d,r)$-arithmetic secret sharing scheme for $\Fq^k$ (over $\Fq$)}
is an $n$-code
$C$ for $\Fq^k$
such that
\begin{itemize}
\item [{\rm (i)}] $t\geq 1$, $d\geq 2$
\item [{\rm (ii)}] $C$  is $t$-disconnected,
\item [{\rm (iii)}] ${C}^{*d}$ is in fact an $n$-code for $\Fq^k$ and
\item [{\rm (iv)}] ${C}^{*d}$ is
$r$-reconstructing.
\end{itemize}
$C$ has {\em uniformity} if, in addition, it is $t$-uniform.
\end{definition}

For example, the case $k=1$, $d=2$, $n=3t+1$, $r=n-t$, $q>n$ obtained from Shamir's secret sharing scheme~\cite{Sha79-A} (taking into account that degrees sum up when taking products
 of polynomials) corresponds to the secret sharing scheme used in~\cite{BGW88,CCD88}. The properties are easily proved using Lagrange's Interpolation Theorem. The
 generalization to $k>1$ of this Shamir-based approach is due to~\cite{FY92}. The abstract notion is due to~\cite{CDM00}, where also constructions for $d=2$ were
 given based on general linear secret sharing. See also \cite{CC06,CCGHV07,CCHP08}. On the other hand the following limitations are easy to establish.
\begin{proposition}
Let $C$ be an $(n,t,d,r)$-arithmetic secret sharing scheme for $\Fq^k$ over $\Fq$.
As a linear secret sharing scheme for $\Fq^k$ over $\Fq$, $C$ has $t$-privacy and $(r-(d-1)t)$-reconstruction. Hence, $dt+k\leq r$. Particularly, if $k=1$, $d=2$, $r=n-t$, then
$3t+1\leq n$.
\end{proposition}

We are now ready to state the asymptotical results from~\cite{CC06} in full generality.\footnote{In fact, we state a version that is proved by exactly the same arguments as in~\cite{CC06}.}
 Let $ F/\Fq$ be an algebraic function field (in one variable, with $\Fq$ as field of constants). Let $g$ denote the genus of $F$. Let $k,t,n\in \ZZ$ with $n>1$,
$1\leq t\leq n$, $1\leq k\leq n$. Suppose
 $Q_1,\dots,Q_k,P_1,\dots,P_n\in \PP^{(1)}(F)$ are {\em pairwise distinct}
$\Fq$-rational places. Write $Q=\sum_{j=1}^k Q_j\in
\divi{}$ and $D= Q+ \sum_{i=1}^n P_i\in
\divi{}$. Let $G\in \divi{}$ be such that $\supp~D\cap\supp~G=\emptyset$, i.e, they have disjoint support.
 Consider the AG-code $C(D,G)_L\subset \FF_q^k\times\FF_q^n$,
given by the image of the map
$$\phi:\mathcal L(G)\rightarrow \FF_q^k\times\FF_q^n$$
$$f\mapsto (f(Q_1), \ldots, f(Q_k), f(P_1), \ldots, f(P_n)).$$

\begin{theorem} (from \cite{CC06}).
 Let $t\geq 1, d\geq 2$. Let $C=C(D,G)_L$ with $deg\ G\geq 2g+t+k-1$. If $n>2dg+(d+1)t+dk-d$, then $C$ is an $(n,t,d,n-t)$-arithmetic sharing scheme for $\Fq^k$ over $\Fq$ with uniformity.
\end{theorem}

\begin{theorem} (from \cite{CC06}).
 Fix  $d\geq 2$  and  a finite field $\Fq$.
 Suppose $A(q)> 2d$, where $A(q)$ is Ihara's constant.
Then there is an infinite family of $(n,t,d,n-t)$-arithmetic secret sharing schemes for $\Fq^k$ over $\Fq$ with uniformity such that
$n$ is unbounded, $k=\Omega(n)$ and $t=\Omega(n)$. Moreover, for every $C$ in the family, a generator for $C$ is
$\mbox{poly}(n)$-time computable and $C^{*i}$ has $\mbox{poly}(n)$-time reconstruction of a secret in the presence of $t$ faulty shares ($i=1,\dots,d-1$).
\end{theorem}

Since $A(q)=\sqrt{q}-1$ if $q$ is a square, it holds that $A(q)>2d$ if $q$ is a square with $q>(2d+1)^2$. Also, since by Serre's Theorem, $A(q)>c \log q$ for some absolute constant $c>0$, it also holds that $A(q)>2d$ if $q$ is (very) large. 
We will now apply our results on the torsion limit in combination with appropriate Riemann-Roch systems in order to relax the condition $A(q)>2d$ considerably. As a result, we attain the result of~\cite{CC06} but this time over {\em nearly all finite fields}.

\begin{theorem}\label{th:genss}
Let $t\geq 1, d\geq 2$. Define  ${\mathcal I}^*=\{1, \ldots, n\}$. For $A\subset {\mathcal I}^*$ with
$A\neq \emptyset$, define $P_A=\sum_{j\in A}P_j\in \divi{}$. Let $K\in
\divi{}$ be a canonical divisor.
 If the system
$$\{  \ell(dX-D+P_A+Q)=0,\ \ell(K-X+P_A+Q)=0\}_ {A\subset{\mathcal I}^*, |A|=t}$$
is solvable, then there is a solution $G\in\divi{}$ such that $C(D,G)_L$ is an  $(n,t,d,n-t)$-arithmetic secret sharing
scheme for $\Fq^k$ over $\Fq$ (with uniformity).
\end{theorem}

The reader may refer \cite{CCX11} for a detailed proof of \ref{th:genss}.

And now as a corollary of Theorems~\ref{thm:system} and~\ref{th:genss} we
get the following:

\begin{corollary}\label{cor:genss2}
Let $F/\Fq$ be an algebraic function field. Let $d,k,t,n\in \ZZ$ with $d\geq 2$, $n>1$
and $1\leq t< n$. Suppose
$Q_1,\dots,Q_k,P_1,\dots,P_n\in \PP^{(1)}(F)$ are pairwise distinct. If there is $s\in \ZZ$
such that
 $$h>\binom{n}{t}(A_{r_1}+A_{r_2}|\mJ_F[d]|)$$ where $r_1:=2g-s+t+k-2$ and
$r_2:=ds-n+t$, then there exists an $(n,t,d,n-t)$-arithmetic secret sharing
scheme for $\Fq^k$ over $\Fq$ with uniformity.

\end{corollary}

\begin{theorem} \label{mth:newresulttor}
Let $\Fq$ be a finite field and $d\in\ZZ_{\geq 2}$. If there exists $0< A\leq A(q)$ such that
$A>d-1+J_d(q,A)$, then there is an infinite family of $(n,t,d,n-t)$-arithmetic secret sharing
 schemes for $\Fq^k$ over $\Fq$
 with $t$-uniformity where $n$ is unbounded, $k=\Omega(n)$ and $t=\Omega(n)$.
\end{theorem}
\begin{remark}
Note that in~\cite[Main Theorem 1]{CCX11} this result was announced but only proved in the case $d=2$. However, 
the general condition is incorrectly written there as $A>1+J_d(q,A)$ instead of $A>d-1+J_d(q,A)$. Note that in the case
 $d=2$ (which is the main concern of~\cite{CCX11}) both expressions
coincide.
\end{remark}

Theorem~\ref{mth:newresulttor} will follow from the more precise statement in Theorem~\ref{thm:newresultexplicit} below.
Combining Theorem~\ref{mth:newresulttor} with Theorem~\ref{2.4a} we obtain,
in the special case $d=2$:

\begin{theorem}\label{mth:newresult}
For $q=8,9$ and for all prime powers $q\ge 16$ there is an infinite family of $(n,t,2,n-t)$-arithmetic secret sharing
 schemes for $\Fq^k$ over $\Fq$
 with $t$-uniformity where $n$ is unbounded, $k=\Omega(n)$ and $t=\Omega(n)$.
\end{theorem}

More precisely, we have the following result (for $d>2$ there is a similar analysis).

\begin{theorem} \label{thm:newresultexplicit}
Let $\Fq$ be a finite field. Suppose $\kappa\in[0,\frac 13)$ and
$\tau\in(0,1]$ and $0< A\leq A(q)$ are real numbers such that
$$A>\frac{1+\kappa}{1-3\kappa}(1+J_2(q,A))$$ and
$$\tau+\frac{H_2(\tau)}{\log
q}<\frac{1}{3}\left(1-3\kappa-\frac{(1+J_2(q,A))(1+\kappa)}{A}\right).$$
Then there is an infinite family of $(n,t,2,n-t)$-arithmetic secret sharing
 schemes for $\Fq^k$ over $\Fq$
with uniformity where $n$ is unbounded, $k=\lfloor\kappa n\rfloor+1$ and $t=\lfloor\tau n\rfloor$.
\end{theorem}

The proof of this fact relies on showing that the conditions in Corollary~\ref{cor:genss2}
are satisfied asymptotically for a family of function field with Ihara's limit $A$, if the
requirements of Theorem~\ref{thm:newresultexplicit} are met. It is easy to show why Theorem~\ref{thm:newresultexplicit} implies  Theorem~\ref{mth:newresult}: if $0< A\leq A(q)$ is such that $A>1+J_2(q,A)$ we
can always select $\kappa\in(0,\frac{1}{3})$ and $\tau\in(0,1]$ satisfying the conditions in Theorem~\ref{thm:newresultexplicit}.
Note that in order to obtain the result in Theorem~\ref{mth:newresult} we require $\kappa>0$. The reader may refer to \cite{CCX11} for the detailed proof of Theorem \ref{thm:newresultexplicit}.

Finally, using our paradigm we also improve the explicit lower bounds for the parameter
$\widehat{\tau}(q)$ from~\cite{CC06} and~\cite{CCCX09} for all $q$ with $q\leq 81$ and $q$
square, as well as for all $q$ with $q\leq 9$. Recall $\widehat{\tau}(q)$ is defined as the maximum value of $3t/(n-1)$ which can be obtained asymptotically
(when $n$ tends to infinity) when $t$, $n$ are subject to the condition that an $(n,t,2,n-t)$-arithmetic secret sharing for $\Fq$ over $\Fq$ exists (no uniformity required here).
 The new bounds are shown in the upper row of Table~1. All the new bounds marked with a star (*) are obtained by applying Theorem~\ref{thm:newresultexplicit}
 in the case $\kappa=0$ and using the upper bounds given in Theorem~\ref{2.2} for the torsion limits. To obtain the rest of the new upper bounds, for each $q$,
 we apply  the field descent technique in~\cite{CCCX09} to $\FF_{q^2}$(in the special case of $\FF_9$, even though Theorem~\ref{thm:newresultexplicit}
 can be applied directly, as remarked in Main Theorem~\ref{mth:newresult}, it is better to apply  Theorem~\ref{thm:newresultexplicit} to $\FF_{81}$ and then use
the descent technique). These are compared with the previous bounds: the ones obtained in~\cite{CC06} (marked also with the symbol *), and the rest, which were
 obtained in~\cite{CCCX09} by means of the aforementioned field descent technique.

\begin{table}[h]\label{tab:boundstau}
\begin{center}

\begin{tabular}{|c|c|c|c|c|c|c|}
\hline
$q$&2&3&4&5&7&8\\
\hline
New bounds&0.034&0.057&0.104&0.107&0.149&0.173$^*$\\
Prev. bounds&0.028&0.056&0.086&0.093&0.111&0.143\\
\hline
\end{tabular}

\begin{tabular}{|c|c|c|c|c|c|c|}
\hline
$q$&9&16&25&49&64&81\\
\hline
New bounds&0.173&0.298$^*$&0.323$^*$&0.448$^*$&0.520$^*$&0.520$^*$\\
Prev. bounds&0.167&0.244&0.278&0.333$^*$&0.429$^*$&0.500$^*$\\
\hline
\end{tabular}

\end{center}
\caption{Lower bounds for $\widehat{\tau}(q)$}
\end{table}

We end this section with the remark that the results above can be adapted to prove a statement about linear codes, namely the existence of
families of codes $C$ such that both their duals and their powers are asymptotically good.

\begin{theorem}
If there exists $0< A\leq A(q)$ such that
$A>d-1+J_d(q,A)$, then there exists an asymptotically good family of linear codes $C$ over $\Fq$ 
such that both the duals $C^{\bot}$ and, for each $1\leq d'\le d$, the powers $C^{*d'}$,
are simultaneously asymptotically good.
In particular, for $d=2$, this holds for $q=8,9$ and for all prime powers $q\geq 16$.
\end{theorem}

In order to show this, we need to adapt the construction of algebraic geometric codes and the proofs of Theorem~\ref{th:genss} and the 
subsequent theorems above. The bottomline is to take the case $k=0$ of those results. Then the same arguments as above prove that both $C^{\bot}$
and $C^{*d}$ have minimum distance linear in the length. The remainder of the claim follows easily from the observations below. 
First, it is easy to show then that for all $1\leq d'\leq d$, the codes $C^{*d'}$ also have minimum distance linear in their length, as it
must be larger than the minimum distance of $C^{*d}$. Second, we have (by Singleton's bound) $\dim C\geq d_{min}(C^{\bot})-1$,
which proves that the codes $C$ are asymptotically good, and analogously we can prove that $C^{\bot}$ are asymptotically good.
Finally, it is obvious that if $d'>d''$ then $\dim C^{*d'}\geq \dim C^{*d''}$, which proves the rest of the statement.

In comparison, if we adapt the results from~\cite{CC06} similarly, we can only prove the existence of
these families of codes under the stronger condition $A(q)>2d$. In the case $d=2$, this means $A(q)>4$, which by the Drinfeld-Vl\v{a}du\c{t}
bound implies $q\geq 25$. The field descent technique based on concatenation of codes from \cite{CCCX09},
which establishes the existence of asymptotically good arithmetic secret sharing over any finite field when no uniformity is required,
does not work here: first, it is not guaranteed that the squares of the resulting codes are asymptotically good and second, the duals cannot be asymptotically good. 
To the best of our knowledge our present paper is the first to establish, for several finite fields $\Fq$, the existence of linear codes $C$ over $\Fq$ such that both $C$, $C^{*2}$ and $C^{\bot}$ are simultaneously asymptotically good.
The existence of such families over $\Fq$ for $2\leq q\leq 13$ is currently an open question except in the cases $q=8,9$. 

Finally, we remark that the case where just $C$, $C^{*2}$ are considered (so the dual $C^{\bot}$ is left out of consideration)
has been shown to hold for all finite fields~\cite{Randriam13-2}, using an algebraic geometric argument in combination
with a refined descent method. The construction applies this field descent method to algebraic geometric codes over a suitable 
extension field such that not only their square but also certain higher powers are asymptotically good. 
The minimum distance of these powers is bounded in~\cite{Randriam13-2} based solely on the degree of the divisors. 
It seems a plausible avenue to try and improve the parameters (dimension, minimum distance)
of the resulting codes $C$ and $C^{*2}$ using the torsion limit but we do not elaborate
further on this here.

\section{Application 2: Bilinear Complexity of
Multiplication}\label{sec:mult}
Since the 1980's, many interesting applications of algebraic curves (or
algebraic function fields of one variable) over finite fields have
been found. One of these applications, which was due to
D.V.~Chudnovsky and G.V.~Chudnovsky~\cite{Chud86}, is the study of
multiplication bilinear complexity in extension fields through algebraic
curves. Following the brilliant work by D.V.~Chudnovsky and
G.V.~Chudnovsky, Shparlinski, Tsfasman and Vl\v{a}du\c{t} \cite{STV92}
systematically studied this idea and extended the result in
\cite{Chud86}. After the above pioneer research, Ballet et al.
\cite{B06,B08,B08-2,BP10} further investigated and developed the
idea and obtained improvements.

Before we formulate the problem, we need to adapt some of the definitions in the previous section.

\begin{definition}

The $\Fq$-vector space morphism
$$\pi_0:\FF_{q^k}\times\Fq^n\rightarrow \FF_{q^k}$$
is defined by the projection  $$(s,c_1,\dots,c_n)\mapsto s.$$

For each $i\in\{1,\dots,n\}$, the $\Fq$-vector space morphism   $$\pi_i:\FF_{q^k}\times\Fq^n\rightarrow \Fq$$ is defined by the projection $$(s, c_1,\dots,c_n)\mapsto c_i.$$
For $\emptyset\neq A\subset\{1,\dots,n\}$, the $\Fq$-vector space morphism $$\pi_A:\FF_{q^k}\times\Fq^n\rightarrow \Fq^{|A|}$$ is defined by the projection  $$(s,c_1,\dots,c_n)\mapsto (c_i)_{i\in A}.$$
For ${\bf v}\in \FF_{q^k}\times\Fq^n$, it is sometimes convenient to denote $\pi_0({\bf v})\in \FF_{q^k}$ by ${\bf v}_0$  and $\pi_A({\bf v})\in \Fq^{|A|}$ by ${\bf v}_A$. We write ${\mathcal I}^*=\{1, \ldots, n\}$.

\end{definition}

\begin{definition} 
 An {\em$n$-code for $\FF_{q^k}$ (over $\Fq$)} is an $\Fq$-vector space $C\subset\FF_{q^k}\times\Fq^n$ such that
\begin{itemize}
\item [{\rm (i)}] $\pi_0(C)=\FF_{q^k}$
\item [{\rm (ii)}] $(\mathrm{Ker\ }\pi_{\mI^{*}})\cap C \ \subset \  (\mathrm{Ker\ }\pi_0)\cap C$.
\end{itemize}

\end{definition}

\begin{definition}

Let $\Fq$ be a finite field, $k>0$ an integer. For two vectors
 ${\bf x}=(x_0,x_1,\dots,x_m), {\bf x}'=(x'_0,x'_1,\dots,x'_m)\in \FF_{q^k}\times \Fq^m$ their
{\em product} ${\bf x}*{\bf x}'\in \FF_{q^k}\times\Fq^m$ is defined as $(x_0x'_0,x_1x'_1, \ldots, x_mx'_m)$
where $x_0x'_0$ is the product in the extension field $\Fq^k$ and $x_ix'_i$ is the product in $\Fq$ for $i=1,\dots,n$.

Let $d$ be a positive integer. If $C$ is a $\Fq$-vector subspace of $\FF_{q^k}\times\Fq^n$, then
$C^{*d}\subset \FF_{q^k}\times \Fq^n$ is the $\Fq$-linear subspace generated by all terms of the form
${\bf c}^{(1)}*\ldots *{\bf c}^{(d)}$ with ${\bf c}^{(1)}, \ldots,  {\bf c}^{(d)}\in C$.
For $d=2$, we use the abbreviation $\widehat{C}:=C^{*2}$.

\end{definition}

Now we can introduce the notion of multiplication-friendly code.
\begin{definition}  Let $n,k\in\ZZ$.
An {\em $(n,k)$-multiplication-friendly code $C$ over $\Fq$} is an $n$-code for $\FF_{q^k}$ (over $\Fq$)
such that
\begin{enumerate}
\item[{\rm (i)}] $n,k\geq 1$.
\item[{\rm (ii)}] $\widehat{C}$ is also an $n$-code for $\FF_{q^k}$.
\end{enumerate}

\end{definition}

\begin{remark}
Since $\pi_0(C)=\FF_{q^k}$ implies $\pi_0(\widehat{C})=\FF_{q^k}$ we can replace $(ii)$ by $$(ii') (x, {\bf 0})\notin \widehat{C}\ \textrm{ for all }\ x\in\FF_{q^k}\setminus\{0\}$$ and we get an equivalent definition.
\end{remark}

Multiplication-friendly codes are also considered in~\cite{STV92}
and are called {\it supercodes} there. By~\cite[Corollary
1.13]{STV92}, an $(n,k)$-{multiplication-friendly code} $C$ over
$\Fq$ yields a bilinear multiplication algorithm of multiplicative
complexity $n$ over $\FF_q$. Therefore, we are interested in the
smallest $n$ for fixed $q$ and $k$.

\begin{definition}We define the quantity 

$$\mu_q(k)=\min_{n\in \ZZ_{>0}} \{n:\textrm{ there exists an
$(n,k)$-multiplication}\\ \textrm{-friendly code over }\Fq\}$$

To measure how $\mu_q(k)$ behaves when $q$ is fixed and $k$ tends to
$\infty$, we define two asymptotic quantities
\[M_q=\limsup_{k\rightarrow\infty}\frac{\mu_q(k)}k\]
and
\[m_q= \liminf_{k\in\N} \frac{\mu_q(k)}{k}.\]
\end{definition}

D.V.~Chudnovsky and G.V.~Chudnovsky~\cite{Chud86} first employed
algebraic curves over finite fields to construct bilinear
multiplication algorithms implicitly through multiplication-friendly
codes in 1986 (please refer to~\cite{BR05_survey} for more
background).
 This idea was further developed in~\cite{STV92} in
order to study the quantities $m_q$ and $M_q$.  The main idea in~\cite{STV92}  is to solve a special Riemann-Roch
 system, stated in  Theorem~\ref{thm:rrmult}. However, the
 role of $2$-torsion points in divisor class group was neglected in~\cite{STV92}, and
 it turns out that there is a gap in  the proof of the main result in~\cite{STV92}. Namely,
 the mistake is in the proof of their Lemma 3.3, page 161, the paragraph following
formulas about the degrees of the divisors. It reads: ``{\it Thus the number of linear equivalence classes
of degree a for which either Condition $\alpha$ or Condition $\beta$ fails is at most $D_{b'}+D_b$}''. This is incorrect.
$D_b$ should be multiplied by the torsion. Hence the proof of their asymptotic bound is incorrect, as
 there is an implicit but (so far) unjustified assumption on $J_2=0$ being possible, or
 rather even the stronger assumption that $\pict{0}{2}=\{0\}$ is possible at all levels in an
 asymptotically good (optimal) family. Therefore, their claim that $m_q\leq 2(1+\frac{1}{A(q)-1})$ is unjustified.
Moreover, some other  results~\cite{B08, B08-2}  use the
same approach and have the same gap
 (the asymptotical results in their precursor~\cite{B06} are based on the conjecture that a tower exists attaining certain properties).
In \cite{B08} the mistake is
at the very beginning of page 1801 (the sentence starts on the previous page):``{\it Hence, the number
of linear equivalence classes of divisors of degree $n + g - 1$ for which either the condition (5) or
the condition (6) fails is at most $2D_{g-1}$ where $D_{g-1}$ denotes...}''. Hence the proof of the asymptotic
bound is incorrect.\\
We will now give an upper bound for $m_q$ which involves the 2-torsion limit introduced in this paper. We first need to state the problem in a way that we can use
the results in Section~\ref{sec:rr}.

\begin{theorem}\label{thm:rrmult}
Let $F/\FF_q$ be an algebraic function field and $N,k>1$ be integers. Suppose there exist
$P_1,\dots,P_N\in \PP^{(1)}(F)$ with $P_i\neq P_j$ ($i\neq j$) and \linebreak $Q\in
\PP^{(k)}(F)$. Let $D=\sum_{i=1}^N P_i+Q\in
\divi{}$ and $D^{-}=\sum_{i=1}^N P_i\in \divi{}$.
Let $K\in \divi{}$ be a canonical divisor.\\
If the
Riemann-Roch system
\[\left\{\begin{array}{c}
\ell(-X+K+Q)=0\\
\ell(2X-D^{-})=0\end{array}
\right.\]
has some solution,
then there exists a solution $G\in \divi{}$ such that\linebreak $\supp~G\cap \supp~D=\emptyset$, and
$C=C(D,G)_L$ is an $(N,k)$-multiplication friendly code over
$\Fq$.\\
Furthermore, write $r=\ell(2G)-\ell(2G-D^{-})$. Then
there exist $r$ indices $i_1,\dots,i_r\in\{1,\dots,N\}$, such that
$\widetilde{C}=C(\widetilde{D},G)_L$ is a $(r,k)$-multiplication-friendly code, where $\widetilde{D}=\sum_{j=1}^r P_{i_j}+Q\in
\divi{}$. Therefore $\mu_q(k)\leq r\leq \ell(2G)$.
\end{theorem}

\begin{proof}
If there exists a solution, any divisor in its class of equivalence is also a solution. By the Weak Approximation Theorem, we can take an element $G$ of this class in such a way that $\supp~ G\cap\supp~ D=\emptyset$.\\

Suppose $G$ is a solution. We prove $C=C(D,G)_L$ is a multiplication-friendly code. We need to verify $\pi_0(C)=\FF_{q^k}$ and $(x, {\bf 0}) \not\in \widehat{C}$ for all $0\neq x\in\FF_{q^k}$.\\

Since $\deg Q=k$, it follows by the Riemann-Roch Theorem and $\ell(K-G+Q)=0$ that $\ell(G)=\ell(G-Q)+k$. This is enough to ensure that $\pi_0(C)=\FF_{q^k}$, as follows:
Consider the map $$\rho:\mL(G)\rightarrow\FF_{q^k},$$ $$f\mapsto f(Q).$$ Its kernel is ${\mL}(G-Q)$. So its image is isomorphic to ${\mL}(G)/{\mL}(G-Q)$, and this has dimension (over $\Fq$) $\ell(G)-\ell(G-Q)=k$. So $\pi_0(C)=\FF_{q^k}$.\\

Second, as $\widehat{C}\subset C(D,2G)_L$, it suffices to prove that $(x, {\bf 0}) \not\in C(D,2G)_L$ for any $0\neq x\in\FF_{q^k}$. Or equivalently,
that any $f\in\mL(2G)$ with $f(P_i)=0$ for $i=1,\dots,N$ satisfies $f(Q)=0$. But this is trivially true as in these conditions, $f\in\mL(2G-D^{-})=\{0\}$.
We have proved $C$ is a multiplication-friendly code.\\

Finally, consider  the $\FF_q$-linear code $C(D^{-},2G)_L$. It has dimension $r$ by definition. Let $i_1,\dots,i_r\in\{1,\dots,N\}$ be such that
the code  $C(\widetilde{D}^{-},2G)_L$ of length $r$ equals $\Fq^{r}$, where $\widetilde{D}^{-}=\sum_{j=1}^r P_{i_j}$.
Note that $\widetilde{C}=C(\widetilde{D},G)_L$ satisfies $\pi_0(\widetilde{C})=\FF_{q^k}$ trivially, since $\pi_0(C)=\FF_{q^k}$ as it is obtained from $C$ by puncturing (``erasing coordinates'') outside the $0$-th coordinate.\\

By construction, $r=\ell(2G)-\ell (2G-\widetilde{D}^{-})$. Since, by definition, it also holds that $r=\ell(2G)-\ell(2G-D^{-})$, it follows that $\mL (2G-D^{-})=\mL (2G-\widetilde{D}^{-})$. So if $f\in {\mL}(2G-\widetilde{D}^{-})$, then $f\in
{\mL}(2G-D^{-})$. This implies $f(Q)=0$, as shown before.
\end{proof}

Combining Theorem~\ref{thm:rrmult} with Theorem~\ref{thm:system}, we get

\begin{theorem}\label{thm:genextmult}
Let $F/\FF_q$ be an algebraic function field and $N,k>1$ be integers. Suppose $|\PP^{(1)}(F)|\geq N$ and
$\PP^{(k)}(F)$ is not empty.
If there is a positive integer $d$ such that $$h>A_{2g-2-d+k}+A_{2d-N}|\mJ[2]|$$
then $\mu_q(k)\leq max \{\ell(2G): G\in \divi{}, \deg G=d\}$.
In particular, if in addition $d\geq g$, then $\mu_q(k)\leq 2d-g+1$.
\end{theorem}
Note that the last part is a consequence of the fact that if $deg~G=d\geq g$, then $deg~2G=2d\geq 2g$ and by Riemann-Roch, $\ell(2G)=2d-g+1$

\begin{theorem}\label{thm:mqbound}
 Let $\Fq$ be a finite field. If there exists a real number $a\leq A(q)$
with $a\geq 1+J_2(q,a)$ then $$m_q\leq 2\left(1+\frac{1}{a-J_2(q,a)-1}\right).$$

In particular, if $A(q)\geq 1+J_2(q,A(q))$, then $$m_q\leq 2\left(1+\frac{1}{A(q)-J_2(q,A(q))-1}\right).$$
\end{theorem}

\begin{proof}
 Let $\mF=\{F_s/\F_q\}_{s=1}^{\infty}$ be an infinite family of function fields with limit $A(\mF)= A\geq a$ and such that $J_2(\mF)=J_2(q,a)$, which exists by definition.
Let $\kappa>0$ be a real number. The precise value of $\kappa$ will be determined later. And define, for every $s$, $g_s=g(F_s)$, $n_s=N_1(F_s)$,
$k_s=\lfloor \kappa g_s\rfloor$ and $j_s=\log_q |J_{F_s}[2]|$. Note $\lim_{s\rightarrow\infty} n_s/g_s=A$ and $\liminf j_s/g_s=J_2(q,a)$.

We will apply \ref{thm:genextmult} to all large enough function fields $F_s$. It is enough to verify that there exists a place $Q$ of degree
$k_s$ in $F_s$ and that \begin{equation}\label{goal}h(F_s)>A_{2g_s-2-d_s+k_s}+|\mJ[2]|A_{2d_s-n_s}\end{equation} holds
for some $d_s$.

First note that \cite[Corollary~5.2.10(c)]{St93} states that
for any function field $F$ and any positive integer $k$ with $q^{(k-1)/2}(q^{1/2}-1)\geq 2g(F)+1$, there is at least one place of degree $k$.
In our setting, since $\lim_{s\rightarrow\infty} k_s/g_s= \kappa>0$, a place of degree $k_s$ exists in $F_s$ for large enough $s$.

 Suppose that for any $\epsilon>0$, there exists a value of $s$ such that
\begin{equation}\label{conditionk}
 k_s\leq \frac{n_s-g_s-j_s}{2}-\epsilon g_s-1.
\end{equation}
Then it is easy to see that we can choose an integer $d_s$ with
\begin{equation}\label{conseqd1}
 d_s\geq k_s+g_s+\frac{\epsilon}{2}g_s
\end{equation}
and
\begin{equation}\label{conseqd2}
 2d_s\leq n_s+g_s-j_s-\epsilon g_s.
\end{equation}

Then for this selection of $d_s$ we can apply Proposition~\ref{propo:Arh} to get

\begin{equation}\label{bound1}
\frac{A_{2g_s-2-d_s+k_s}}{h}\leq \frac{g_s}{q^{g_s-(2g_s-2-d_s+k_s)-1}(\sqrt{q}-1)^2}
\end{equation}

and

\begin{equation}\label{bound2}
|\mJ[2]|\frac{A_{2d_s-n_s}}{h}\leq \frac{g_sq^{j_s}}{q^{g_s-(2d_s-n_s)-1}(\sqrt{q}-1)^2}
\end{equation}

Now if $s$ is large enough, equations~\ref{conseqd1} and~\ref{bound1} imply that  $$A_{2g_s-2-d_s+k_s}\leq h/3$$ and
equations~\ref{conseqd2} and~\ref{bound2} imply that $$|\mJ[2]|A_{2d_s-n_s}\leq h/3,$$ so equation~\ref{goal} holds, and we
can apply Theorem~\ref{thm:genextmult} and (since in addition $d_s\geq g_s-1$ by equation~\ref{conseqd1}), this gives $\mu_q(k_s)\leq 2d_s-g_s+1$.
In particular, since we can take $\epsilon$ arbitrarily small, we can choose $d_s= k_s+g_s+1$, and this yields the bound $\mu_q(k_s)\leq 2k_s+g_s+3$.

So all is left is to determine when we can fulfill condition~\ref{conditionk}. It is not difficult to see that if $\kappa<\frac{A-1-J_2(q,a)}{2}$, then for an infinite number
of values of $s$, and for small enough (but constant) $\epsilon$, the condition holds.

Therefore, for those values of $s$, we have
$$\frac{\mu_q(k_s)}{k_s}\leq \frac{ 2k_s+g_s+3}{k_s}\leq \frac{(2+\frac{1}{\kappa})k_s+o(1)}{k_s}\rightarrow 2+\frac{1}{\kappa}$$ for any $\kappa<\frac{A-1-J_2(q,a)}{2}$.

Hence 
$$m_q=\liminf_{k\rightarrow\infty}\frac{\mu_q(k)}{k}\leq 2+\frac{2}{A-1-J_2(q,a)}\leq\\ 2\left(1+\frac{1}{a-J_2(q,a)-1}\right) $$

which finishes the proof.
\end{proof}

\begin{remark}\label{rem:randriam}
Recently in \cite{Randriam}, H. Randriambololona proved that the original result claimed in \cite{STV92}, i.e. $m_q\leq 2(1+\frac{1}{A(q)-1})$, can indeed be attained in the case $A(q)>5$.
\footnote{Note that in \cite{Randriam}, our notion $m_q$ is denoted by $m_q^{sym}$.}
\end{remark}

From Theorem~\ref{2.4a}, we can apply Theorem~\ref{thm:mqbound} to all fields $\Fq$ with $q\geq 8$, except perhaps $q=11$ and $13$.
These include several fields for which the result in Remark~\ref{rem:randriam} cannot be applied directly. However, we must also take into account
the following descent lemma which, combined with any of these results, allows to obtain upper bounds for $m_q$ for all fields $\Fq$.

\begin{lemma}\cite[Corollary 1.3]{STV92}\label{lemma:descentmult}
 For every finite field $\Fq$ and every positive integer $k$, we have $$m_q\leq \frac{\mu_q(k)}{k} m_{q^k}.$$
\end{lemma}

In order to obtain explicit results, we need some values of $\mu_q(k)$ for small values of $k$. We can use the following lemma,
which for example can be found in \cite[Example III.5]{CCXY11}.

\begin{lemma}\cite[Example III.5]{CCXY11}
 Let $q$ be a prime power and $k$ be an integer with $2\leq k\leq q/2+1$. Then $\mu_q(k)=2k-1$.
In particular $\mu_q(2)=3$ for every $q$ and $\mu_q(3)=5$ for every $q\geq 4$.
\end{lemma}
\begin{corollary}\label{cor:descentmq}
 For every prime power $q$, we have $m_q\leq \frac{3}{2} m_{q^2}$ and if $q\geq 4$, then $m_q\leq \frac{5}{3} m_{q^3}$.
\end{corollary}

These observations allow us to compare the bounds which result from Theorem~\ref{thm:mqbound} with those implied by the result in
Remark~\ref{rem:randriam}. We find then that our Theorem~\ref{thm:mqbound} gives the best bound in the cases $q=16,\ 25,\ 32$
while for the rest of cases, applying Remark~\ref{rem:randriam} in a suitable extension and then using the descent results above
is preferable, given the current knowledge about $A(q)$ and the bounds for the torsion limit given in Theorem~\ref{2.2}. We give some examples in Table~2.
For $q=8,\ 9,\ 27$, the results are found by applying Theorem~\ref{thm:mqbound} and Remark~\ref{rem:randriam} to $\FF_{q^2}$
(followed by Corollary~\ref{cor:descentmq}). Note in particular that it would be possible to apply Theorem~\ref{thm:mqbound} directly in these
cases, yet it would give a worse bound. For $q=4,\ 5$, we apply Theorem~\ref{thm:mqbound} and Remark~\ref{rem:randriam} to $\FF_{q^3}$.
For $q=2,\ 3$ we use the bounds for $m_{q^2}$ that we have just computed. Finally, for $q=16,\ 25,\ 32$ we apply Theorem~\ref{thm:mqbound}
directly on $\FF_q$, while we apply Remark~\ref{rem:randriam} on $\FF_{q^2}$. For the case $q=16$, the
fact that we can prove an improved torsion bound (we are in the case (iii) of Theorem~\ref{2.2}) using the theorem
of Deuring-Shafarevich is significant, as otherwise we would only be able to prove the bound $m_{16}\leq 3.334$ this way.

\begin{table}[h]\label{tab:boundsmq}
\begin{center}

\begin{tabular}{c|c|c|c|c|c|}
$q$&2&3&4&5&8\\
\hline
Thm.~\ref{thm:mqbound}&5.836&5.174&3.891&3.932&3.501\\
Rem.~\ref{rem:randriam}&5.834&5.143&3.889&3.903&3.5\\
\hline
$q$&9&16&25&27&32\\
\hline
Thm.~\ref{thm:mqbound}&3.449&{\bf 3.026}&{\bf 2.779}&3.121&{\bf 2.667}\\
Rem.~\ref{rem:randriam}&3.429&3.215&3.131&3.12& 3.1\\
\hline
\end{tabular}
\caption{Upper bounds for $m_q$}
\end{center}
\end{table}

In the rest of this section, we improve the state of the art~\cite{CCXY11} regarding lower bounds on the limit $M_q$, for
\emph{small} values of $q$ such as $q=2,3,4,5$. The following result can be found in \cite{CCXY11}.
\begin{proposition}\label{splitting} Let $F/\F_q$ be a function field with $r$
distinct places $P_1,\dots,P_r$. Let $Q$ be a place of degree $k$.
If there exists a divisor $G$
 such that the following two conditions
are satisfied \begin{itemize} \item[{\rm (i)}]
$\ell(G)-\ell(G-Q)=\deg(Q)$;\item[{\rm (ii)}]
$\ell(2G-\sum_{i=1}^rP_i)=0$
\end{itemize} then \[\mu_q(k)\le \sum_{i=1}^r\mu_q(s_i),\]
 where  $s_i=\deg(P_i)$ for all $1\le i\le r$.
\end{proposition}

The two conditions of Proposition \ref{splitting} can be replaced by the solvability of certain Riemann-Roch system as shown below.

\begin{corollary}\label{split} Let $F/\F_q$ be a function field with $r$
distinct places $P_1,\dots,P_r$. Let $Q$ be a place of degree $k$.
If the Riemann-Roch system \[\left\{\begin{array}{c}
\ell(K-X+Q)=0\\
\ell(2X-\sum_{i=1}^rP_i)=0\end{array}
\right.\]
has solutions for a canonical divisor $K$, then \[\mu_q(k)\le \sum_{i=1}^r\mu_q(s_i),\]
 where  $s_i=\deg(P_i)$ for all $1\le i\le r$.
\end{corollary}
\begin{proof} Suppose that $G$ is a solution. Then we have $\cL(K-G+Q)=0$, and hence $\cL(K-G)=0$. Thus, we have

\[\ell(G)-\ell(G-Q)=\deg(Q)+\ell(K-G)-\ell(K-G+Q)=\deg(Q).\]
The desired result follows from Proposition \ref{splitting}.
\end{proof}

Now combining Corollary \ref{split} with Theorem \ref{thm:system}, we obtain a numerical condition.
\begin{theorem}\label{thm:splitting}
Let $F/\F_q$ be a function field with $r$
distinct places $P_1,\dots,P_r$. Let $Q$ be a place of degree $k$. Denote by $A_r$ the number of effective divisors of degree $r$ in $\Div(F)$.
If there is a positive integer $d$ such that $$h>A_{2g-2-d+k}+|\mJ[2]|A_{2d-\sum_{i=1}^r s_i},$$ then
\[\mu_q(k)\le \sum_{i=1}^r\mu_q(s_i),\]
 where  $s_i=\deg(P_i)$ for all $1\le i\le r$.
\end{theorem}

To derive a lower bound on $M_q$, we need a family of Shimura curves with
genus in this family growing slowly (see \cite[Lemma IV.4]{CCXY11}).

\begin{lemma}\label{lem:slowgenus}
 For any prime power $q$ and integer $t\ge
1$, there exists a family $\{\cX_s\}_{s=1}^{\infty}$ of Shimura curves over $\F_q$ such that
\begin{itemize}
\item[{\rm (i)}] The genus $g(F_s)\rightarrow\infty$ as $s$ tends
to $\infty$, where $F_s$ stands for the function field $\F_q(\cX_s)$. \item[{\rm (ii)}]
$\lim_{s\rightarrow\infty}g(F_s)/g(F_{s-1})=1$. \item[{\rm (iii)}]
$\lim_{s\rightarrow\infty}B_{2t}(F_s)/g(F_s)=(q^t-1)/(2t)$, where
$B_{2t}(F_s)$ stands for the number of places of degree $2t$ in
$F_s$.
\end{itemize}
\end{lemma}

Now we are ready to derive the following result.

\begin{theorem}\label{thm:upperbound} For a prime power $q$, one has
\[M_q\le \left\{\begin{array}{cc}
\mu_q(2t)\frac{q^t-1}{t(q^t-2-\log_q2)}& \mbox{if $2|q$}\\
\mu_q(2t)\frac{q^t-1}{t(q^t-2-2\log_q2)} & \mbox{otherwise}
\end{array}\right.\]
for any $t\ge 1$ as long as $q^t-2-\log_q2>0$ for even $q$; and $q^t-2-2\log_q2>0$ for odd $q$.
\end{theorem}
\begin{proof} We prove the theorem only for the case where $q$ is a power of $2$. For the odd characteristic case, the only difference is the size of $\J[2]$.

 Let $\{F_s/\F_q\}_{s=1}^{\infty}$ be a family of function
fields with the three properties in Lemma \ref{lem:slowgenus}. For every
$k\ge 2$, let $s(k)$ be the smallest positive integer such that
\begin{equation}\label{eq4.1}
B_{2t}(F_{s(k)})\ge r:=\left\lceil\left(\frac12g_{s(k)}(1+\log_q2)+k+\right.\right.\\
\left.\left.\frac32\log_q\left(\frac{3qg_{s(k)}}{(\sqrt{q}-1)^2}\right)+1\right)/t\right\rceil,
\end{equation}
 where
$g_{s(k)}$ is the genus $g(F_{s(k)})$ of $F_{s(k)}$.

Thus, we can find $r$ places of degree $2t$ in  $F_{s(k)}$.
By the definition of $r$ in Equation (\ref{eq4.1}), we have
 \begin{equation}\label{eq4.2}
g_{s(k)}+k+\log_q\left(\frac{3qg_{s(k)}}{(\sqrt{q}-1)^2}\right)\le \frac12g_{s(k)}(1-\log_q2)+\\
rt-\frac12\log_q\left(\frac{3qg_{s(k)}}{(\sqrt{q}-1)^2}\right)-1.
\end{equation}
Therefore, we can find an integer $d$ between $g_{s(k)}+k+\log_q\left(\frac{3qg_{s(k)}}{(\sqrt{q}-1)^2}\right)$ and $\frac12g_{s(k)}(1-\log_q2)+rt-\frac12\log_q\left(\frac{3qg_{s(k)}}{(\sqrt{q}-1)^2}\right)$, i.e., we have
 \begin{equation}\label{eq4.3}
\frac{g_{s(k)}}{q^{g_{s(k)}-(2g_{s(k)}-d+k)-1}(\sqrt{q}-1)^2}\le \frac 13
\end{equation}
and
 \begin{equation}\label{eq4.4}
\frac{g_{s(k)}2^{g_{s(k)}}}{q^{g_{s(k)}-(2d-2rt)-1}(\sqrt{q}-1)^2}\le \frac 13
\end{equation}
Using the fact that $|\mJ[2]|\le q^{g_{s(k)}}$ and combining Equations (\ref{eq4.3}), (\ref{eq4.4}) and Proposition~\ref{propo:Arh}, we get
\[h>\frac{2h}3\ge A_{2g_{s(k)}-d+k}+|\mJ[2]|A_{2d-2rt},\]
where $h$ is the zero divisor class number of $F_{s(k)}$.
By Theorem \ref{thm:splitting}, we have \[\mu_q(k)\le r\mu_q(2t).\]
On the other hand, by choice of $s(k)$, we know that
\begin{equation}\label{eq4.5}
B_{2t}(F_{s(k)-1})\le\left\lceil\left(\frac12g_{s(k)-1}(1+\log_q2)+k+\right.\right.\\
\left.\left.\frac32\log_q\left(\frac{3qg_{s(k)-1}}{(\sqrt{q}-1)^2}\right)+1\right)/t\right\rceil-1,
\end{equation}
By the property (iii) in Lemma \ref{lem:slowgenus}, the inequality
(\ref{eq4.5}) gives
\begin{equation}\label{eq4.6}
k\ge \frac12g_{s(k)-1}(q^t-2-\log_q2)+o(g_{s(k)-1}).\\
\end{equation}
Finally by Theorem \ref{thm:splitting}, we have
$$\frac{\mu_q(k)}k\le\frac{ r\mu_q(2t)}k\le\mu_q(2t)\left(\frac{(1+\log_q2)g_{s(k)}+o(g_{s(k)})}{2kt}+\frac1t\right)
$$
$$=\mu_q(2t)\left(\frac{(1+\log_q2)g_{s(k)}+o(g_{s(k)})}{t(g_{s(k)-1}(q^t-2-\log_q2)+o(g_{s(k)-1}))}+\frac1t\right)
$$
$$\rightarrow \mu_q(2t)\frac{q^t-1}{t(q^t-2-\log_q2)}\quad \mbox{as } k\rightarrow\infty.
$$

This finishes the proof.
\end{proof}
Note that in \cite{CCXY11},  a trivial solution of the Riemann-Roch system in Corollary \ref{split} was used due to the fact that torsion limit was not considered, and hence a weaker bound on $M_q$ was derived in \cite{CCXY11}.

With help of the torsion-limit technique and Riemann-Roch system, we can bring down the upper bound derived in Theorem \cite[Theorem IV.5]{CCXY11} and hence we get further improvements on $M_q$ for small values of $q$. Here we only provide upper bounds for a few small $q$ to demonstrate our improvements.
\begin{corollary}\label{cor:table} One has the upper bounds on $M_q$
for $q=2,3,4,5$ as shown in the following table
\begin{center}
\begin{tabular}{|c||c|c|c|c|}\hline
$q$&$2$&$3$&$4$&$5$\\ \hline $M_q$&$7.23$&$5.45$&$4.44$&$4.34$\\
\hline
\end{tabular}
\end{center}
\end{corollary}
\begin{proof}
\begin{itemize}\item[(i)] For $q=2$, the desired result follows from Theorem \ref{thm:upperbound} by taking $t=6$ and applying
$\mu_2(12)\le42$.
\item[(ii)] For $q=3$, the desired result follows from Theorem \ref{thm:upperbound} by taking $t=5$ and applying
$\mu_3(10)\le27$.
\item[(iii)] For $q=4$, the desired result follows from Theorem \ref{thm:upperbound} by taking $t=2$ and applying
$\mu_4(4)=8$.
\item[(iv)] For $q=5$, the desired result follows from Theorem \ref{thm:upperbound} by taking $t=2$ and applying
$\mu_5(4)=8$.
\end{itemize}
\end{proof}

\section{Application 3: Asymptotic Bounds for Frameproof Codes}\label{sec:frame}
\subsection{Definitions and basic results}
Let $S$ be a finite set of $q$ elements (we denote by $\F_q$ the finite field with $q$ elements if $q$ is a prime power) and let $n$ be a positive integer. Define the $i$-th projection:
\[\pi_i: \ S^n\rightarrow S,\quad (a_1,\dots,a_n)\mapsto a_i.\]
\begin{definition}
 For a subset $A\subset S^n$, we define the {\it descendants} of $A$, desc$(A)$, to be the set of all words $\bx$ such that
 for each $1\le i\le n$, there exists $\ba\in A$ satisfying $\pi_i(\bx-\ba)=0$.
\end{definition}
\begin{definition}
 Let $s\ge 2$ be an integer. A {\it $q$-ary $s$-frameproof code} of length $n$ is a subset $C\subset S^n$ such that for all $A\subset C$ with $|A|\le s$,
the intersection desc$(A)\cap C$ is the same as $A$.
\end{definition}
Note that $1$-frameproof codes are uninteresting, since any $C\subset S$ would satisfy the resulting condition.
From the definition of frameproof codes, it is clear that a $q$-ary $s$-frameproof code $C$ is a $q$-ary $s_1$-frameproof code for any $2\le s_1\le s$.

Following the notation from \cite{SW98}, we denote a $q$-ary $s$-frameproof code in $S^n$ of size $M$ by $s$-$FPC(n,M)$. As usual, we denote a $q$-ary error-correcting code of length $n$, size $M$ and minimum distance $d$ by $(n, M, d)$-code, or $[n,\log_qM,d]$-linear code if the code is linear.

We want to look at the asymptotic behavior of $s$-frameproof codes in the sense that $q$ and $s$ are fixed and the length $n$ tends to infinity.

\begin{definition}\label{def:frameproof}
 For  fixed integers $q\ge 2$, $s\ge 2$ and $n\ge 2$, let $M_q(n,s)$ denote the maximal size of $q$-ary $s$-frameproof codes of length $n$, i.e,
\[M_q(n,s):=\max\{M: \ \mbox{there exists a $q$-ary $s$-$FPC(n,M)$}\}.\]
For fixed $q$ and $s$, define the asymptotic quantity
\[R_q(s)=\limsup_{n\rightarrow\infty}\frac{\log_qM_q(n,s)}n.\]
\end{definition}

It seems that the exact values of $R_q(s)$ are not easy to be determined
for any given $q$ and $s$. Instead, we will get some lower bounds on $R_q(s)$.
Before looking at lower bounds, we first derive an upper bound on $R_q(s)$ from
\cite{B01}.

\begin{theorem}\label{thm:frameproofupper} \[R_q(s)\le \frac 1s.\]
\end{theorem}
  \begin {proof} By Theorem 1 of \cite{B01}, we have
\[ M_q(n,s)\le \max\{q^{\lceil\frac ns\rceil}, r\left(
q^{\lceil\frac ns\rceil}-1\right)+(s-r)\left(q^{\lfloor\frac ns\rfloor}-1\right)\},\]
where $r\in\{0,1,\dots,s-1\}$
and $r$ is the remainder of $n$ divided by $s$. Thus, we have
\[M_q(n,s)\le sq^{\lceil\frac ns\rceil}.\]
The desired result follows.
\end {proof}

From now on we will concentrate on lower bounds on $R_q(s)$. Let us first recall the constructions from \cite{CE00}.
\begin{proposition}\label{prop:frameprooffromcode} Let $q$ be a prime power. Then a
$q$-ary $[n,k,d]$-linear code $C$ is a $q$-ary $s$-$FPC(n,q^k)$ with $s=\lfloor (n-1)/(n-d)\rfloor$.
\end{proposition}

\begin{remark} This construction shows that the crucial parameter $s$
is determined only by the minimum distance of $C$ if the length is given.
\end{remark}

From the above relationship between linear codes and frameproof codes, we immediately obtain a lower bound on $R_q(s)$ from the Gilbert-Varshamov bound.

\begin{theorem}\label{thm:frameprooflower}
Let $q$ be a prime power and $2\le s< q$ an integer. Then
\[R_q(s)\ge 1-H_q\left(1-\frac 1s\right),\]
where
\[H_q(\delta)=\delta\log_q(q-1)-\delta\log_q\delta-\left(1-\delta)\log_q(1-\delta\right)\]
is the $q$-ary entropy function.
\end{theorem}
\begin{proof}

The desired result follows directly from the Gilbert-Varshamov bound and
Proposition~\ref{prop:frameprooffromcode}.
\end{proof}

\begin{remark} The bound in Theorem~\ref{thm:frameprooflower} is only an existence result as the Gilbert-Varshamov bound is not constructive.
\end{remark}

\subsection{Lower Bounds from AG Codes}
In this section, we introduce two lower bounds on $R_q(s)$ from algebraic geometry codes. One bound can be obtained  by directly applying Proposition~\ref{prop:frameprooffromcode} and the Tsfasman-Vl\u{a}du\c{t}-Zink bound~\cite{TVZ82}.
However, the second bound employs our torsion limits.

\begin{theorem}\label{thm:fpcasbound}
For a prime power $q$ and  an integer $s\ge 2$, we have
\[R_q(s)\ge \frac 1s-\frac 1{A(q)}.\]
\end{theorem}
\begin{proof} Let $\delta=1- 1/s$. Combining Proposition \ref{prop:frameprooffromcode}
with the TVZ bound, we obtain the desired result.
\end{proof}

\begin{remark} \begin{itemize} \item[(i)] The bound in Theorem \ref{thm:fpcasbound} is constructive  as long as sequences of curves attaining $A(q)$ are explicit.
\item[(ii)] It is easy to check that for every $s\ge 2$, the bound in Theorem \ref{thm:fpcasbound} is better than the one in Theorem \ref{thm:frameprooflower} for sufficiently large square $q$. For instance, for $s=2$, and a square $q\ge 49$,  the bound in Theorem \ref{thm:fpcasbound} is always better than the one in Theorem \ref{thm:frameprooflower}.
\item[(iii)] Comparing with the upper bound in Theorem \ref{thm:frameproofupper}, we find that
\[\frac 1s-\frac 1{A(q)}\le R_q(s)\le \frac 1s.\]
Since  $1/A(q)\rightarrow 0$ as
$q\rightarrow\infty$ (see \cite{NX01}),  $R_q(s)$ is getting closer to $1/s$ as  $q\rightarrow\infty$. The result $R_q(s)\approx 1/s$ is also implicitly
stated in \cite{CE00} by combining Propositions 2 and 3 there.
\end{itemize}

\end{remark}

The bound in  Theorem \ref{thm:fpcasbound} has been further improved in \cite{X02,Randriam10,Randriam13}.

\begin{theorem}\label{XR} \begin{itemize}
\item[(i)]\cite{X02} For every $2\le s\le A(q)$, one has
\[R_q(s)\ge \frac 1s-\frac 1{A(q)}+\frac{1-2\log_qs}{sA(q)}.\]
\item[(ii)]\cite{Randriam10} Let $s$ be the characteristic of $\F_q$, then one has
\[R_q(s)\ge \frac 1s-\frac 1{A(q)}+\frac{1-\log_qs}{sA(q)}.\]
\item[(iii)]\cite{Randriam13} For  $A(q)>5$, one has
\[R_q(2)\ge\frac12-\frac{1}{2A(q)}.\]
\end{itemize}
\end{theorem}

For the rest of this section, we derive a lower bound on $R_q(s)$ by making use of the idea from \cite{X02} and our torsion limit. In particular,  the bounds (i) and (ii) of Theorem \ref{XR} can be deduced from our lower bound in Theorem \ref{3.8}. Furthermore, we improve the above bounds in the following two cases: (i) when $q$ is a square and $s$ is the characteristic of $\F_q$, the bound in Theorem \ref{XR}(ii) can be improved  significantly (see Corollary \ref{FPC}(i)); (ii) when $s$ does not divide $q-1$, the bound in Theorem \ref{XR}(i) can be improved (see Corollary \ref{FPC}(ii)).

Let $P_1,P_2,\dots,
P_n$ be $n$ distinct rational points of a function field $F$ over the finite field $\Fq$. Choose a positive divisor
$G$  such that $\mL(G-\sum_{i=1}^nP_i)=\{0\}$.
Let $\nu_{P_i}(G)=v_i\ge 0$ and $t_i$ be a local parameter at $P_i$ for each $i$.

Consider the map
$$\phi \ : \ \mL(G)\longrightarrow \F_q^n$$ $$
f\mapsto ((t_1^{v_1}f)(P_1),(t_2^{v_2}f)(P_2),\dots,(t_n^{v_n}f)(P_n)).$$
Then the image of $\phi$ forms a subspace of $\F_q^n$ that
is defined as an algebraic geometry code.
The image of $\phi$ is denoted by $C(\sum_{i=1}^nP_i,G)_L$.
The map $\phi$ is an embedding since $\mL(G-\sum_{i=1}^nP_i)=\{0\}$
and the dimension  of $C(\sum_{i=1}^nP_i,G)_L$ is equal to  $\ell(G)$.

\begin{remark} Notice that the above construction is a modified version
of algebraic geometry codes defined by Goppa. The advantage of the above construction is to make it possible to get rid of the condition
${\rm Supp}(G)\cap\{P_1,P_2,\dots,P_n\}=\emptyset$. This is crucial for our construction of frameproof codes in this section.

When the condition ${\rm Supp}(G)\cap\{P_1,P_2,\dots,P_n\}=\emptyset$ is satisfied, i.e., $v_i=0$ for all $i=1,\cdots,n$, then the above construction of algebraic geometry codes is consistent with Goppa's construction.
\end{remark}

\begin{theorem}\label{3.2a}
Let $F/\F_q$ be
an algebraic function field of genus $g$ and let $P_1,P_2,\dots,
P_n$ be $n$ distinct rational points of $F$. Let $G$ be a positive divisor
 such that $\deg(G)<n$. Let $s\ge 2$ satisfy $\mL(sG-\sum_{i=1}^nP_i)=\{0\}$. Then $C(\sum_{i=1}^nP_i,G)_L$ is an $s$-$FPC(n,q^{\ell(G)})$.
\end{theorem}
\begin{proof}
For all $f\in\L(G)$, denote by $\cc_f$ the codeword
\[\phi(f)=((t_1^{v_1}f)(P_1),(t_2^{v_2}f)(P_2),\dots,(t_n^{v_n}f)(P_n)).\]
Let $A=\{\cc_{f_1},\dots,\cc_{f_r}\}$ be a subset of $C:=C(\sum_{i=1}^nP_i,G)_L$ with $|A|=r\le s$.
Let $\cc_g\in \d(A)\cap C$ for some $g\in \L(G)$. Then by the definition of descendant,
for each $1\le i\le n$ we have
\[\prod_{j=1}^r\pi_i(\cc_{f_j}-\cc_g)=0,\]
where $\pi_i(\cc_{f_j}-\cc_g)$ stands for  $i$th coordinate of $\cc_{f_j}-\cc_g$. This implies that
\[\prod_{j=1}^r(t_i^{v_i}f_j-t_i^{v_i}g)(P_i)=0,\]
i.e.,
\[\nu_{P_i}(\prod_{j=1}^r(t_i^{v_i}f_j-t_i^{v_i}g))\ge 1.\]
This is equivalent to
\[\nu_{P_i}(\prod_{j=1}^r(f_j-g))\ge -rv_i+1.\]
Hence,
\[\prod_{j=1}^r(f_j-g)\in\L(rG-\sum_{i=1}^nP_i)\subset\L(sG-\sum_{i=1}^nP_i)=\{0\}.\]
Thus, the function  $\prod_{j=1}^r(f_j-g)$ is the zero function. So, $f_l-g=0$ for some $1\le l\le r$. Hence
$\cc_g=\cc_{f_{l}}\in A$.
\end{proof}

From Theorem \ref{3.2a}, we know that it is crucial to find a
divisor $G$ such that $\L(sG-\sum_{i=1}^nP_i)=\{0\}$.
Again we can apply our Theorem~\ref{thm:system} to show

\begin{lemma}\label{lem:condframe}
Let $F/\F_q$ be an algebraic function field of genus $g$ with at least one rational point $P_0$.

Let $s, m,n$ be three integers satisfying  $s\ge 2$ and $g\le m\le
n<sm$ and $H$ a fixed positive divisor of degree $n$. Then there
exists a positive divisor $G$ of degree $m$  such that
$\L(sG-H)=\{0\}$
provided that
 $A_{sm-n}|\mJ[s]|<h$.
\end{lemma}

\begin{lemma}\label{lem:condframe2}

Let $F/\F_q$ be an algebraic function field of genus $g$ with at least one rational point.

Let $s, m,n$ be three integers satisfying  $s\ge 2$ and $g\le m\le
n<sm$ and  $sm-n< g-\log_q|\mJ[s]|-\log_q\frac{qg}{(\sqrt{q}-1)^2}$. Let $D$ be a
fixed positive divisor of degree $n$. Then there exists a positive
divisor $G$ of degree $m$ such that
$\L(sG-D)=\{0\}$.
\end{lemma}

\begin{proof}
 By Proposition~\ref{propo:Arh} we have (note $1\leq sm-n\leq g-1$)
\[\frac{A_{sm-n}}{h}\le\frac{g}{q^{g-(sm-n)-1}(\sqrt{q}-1)^2}.\]
The condition in Lemma~\ref{lem:condframe} is satisfied and the desired result
follows.
\end{proof}

\begin{theorem}\label{3.8}
Suppose that $q$ is a prime power and $s$ is an integer such that
$A(q)\ge s\ge 2$ and $J_s(q,A(q))<1$. Then we have
\[R_q(s)\ge \frac 1s-\frac 1{A(q)}+\frac{1-J_s(q,A(q))}{sA(q)}.\]
\end{theorem}
\begin{proof}
 Choose a family of function fields $F/\F_q$  with growing genus such that
 $\lim_{g(F)\rightarrow\infty}
N(F)/g(F)=A(q)$ and $\lim_{g(F)\rightarrow\infty}
\log_q|\mJ[s]|/g(F)=J_s(q,A(q))$. Put $n= N(F)$, $g=g(F)$. Let
$D=\sum_{P\in\PP^{(1)}(F)}P.$

 Now for any fixed
$0<\varepsilon<1-J_s(q,A(q))$, put
\[m=\lfloor\frac {n+(1-J_s(q,A(q))-\varepsilon )g}{s}\rfloor.\]
Then we obtain $$
\lim_{g\rightarrow\infty}\frac{m}{g}=\frac{A(q)+1-J_s(q,A(q))-\varepsilon}s>
\frac{A(q)}{s}\ge 1, $$ and 
$$
\lim_{n\rightarrow\infty}\frac{m}{n}=\frac{A(q)+1-J_s(q,A(q))-\varepsilon}{sA(q)}<
\frac{A(q)+1}{sA(q)}<\frac {2A(q)}{sA(q)}\le 1, 
$$
 and $$
\lim_{n\rightarrow\infty}\frac{sm}{n}=1+\frac{1-J_s(q,A(q))-\varepsilon}{A(q)}>1,
$$ and
$$\lim_{n\rightarrow\infty}\frac{sm-n-(1-J_s(q,A(q)))g}{g}=-\varepsilon<0.$$
Therefore, for all sufficiently large $g$ we have $g\le m<n<sm$ by
(2), (3) and (4). It follows from (5) that for all sufficiently
large $g$ we have
\[
sm-n< g-\log_q|\mJ[s]|-\log_q\frac{qg}{(\sqrt{q}-1)^2}.\] By Lemma
\ref{lem:condframe2}, there exists a divisor $G$ of degree $m$ of $F$ such
that $\L(sG-D)=\{0\}$ for each sufficiently large $g$. Thus, by
Theorem \ref{3.2a} the code $C(D,G)_L$ is an $s$-$FPC(n,q^{\ell(G)})$.
Hence,
\begin{eqnarray*}
R_q(s)&\ge&\lim_{g\rightarrow\infty}\frac{\log_qq^{\ell(G)}}{n}\\
&\ge&\lim_{g\rightarrow\infty}\frac{{m-g+1}}{n}\\
&=&\frac 1s-\frac 1{A(q)}+\frac{1-J_s(q,A(q))}{sA(q)}-\frac{\varepsilon}{sA(q)}.\\
\end{eqnarray*}
Since the above inequality holds for any $0<\varepsilon<1-J_s(q,A(q))$,
we get
\[R_q(s)\ge \frac 1s-\frac 1{A(q)}+\frac{1-J_s(q,A(q))}{sA(q)}\]
by letting $\varepsilon$ tend to $0$. This completes the proof.
\end{proof}

\begin{corollary}\label{FPC} Suppose that $q$ is a prime power and $s$ is an integer such that
$A(q)\ge s\ge 2$. Then we have
\begin{equation}\label{eq:L1}R_q(s)\ge \frac 1s-\frac 1{A(q)}+\frac{1-2\log_qs}{sA(q)}.\end{equation}
Moreover, we obtain an improvement to the bounds in Theorem \ref{XR}  for the following two cases.
\begin{itemize}
\item[{\rm (i)}] If $q$ is a square and $s$ is the characteristic of
$\F_q$ with  $\sqrt{q}-1\ge s\ge 2$, then
\begin{equation}\label{eq:L2}R_q(s)\ge \frac 1s-\frac 1{\sqrt{q}-1}+\frac{(1-(\log_qs)/(\sqrt{q}+1))}{s(\sqrt{q}-1)}.\end{equation}
\item[{\rm (ii)}] If $s$ does not divide $q-1$, then
\begin{equation}\label{eq:L3}R_q(s)\ge \frac 1s-\frac 1{A(q)}+\frac{1-\log_qs}{sA(q)}.\end{equation}
\end{itemize}
\end{corollary}
\begin{proof} The  bounds (\ref{eq:L1}), (\ref{eq:L2}) and (\ref{eq:L3}) follow from Theorems \ref{3.8} and Theorem \ref{2.2}(i),  \ref{2.2}(iii) and \ref{2.2}(ii), respectively.
\end{proof}

\section{Acknowledgments}

We are grateful for valuable contributions to the refinements on the bounds for the torsion limit in Theorem~\ref{2.2}.
Bas Edixhoven and Hendrik Lenstra suggested  the generic approach we used in its second part. Alp Bassa and Peter Beelen
confirmed our hope that stronger bounds should be attainable from certain {\em specific} recursive towers, by contributing the proof of its third part. We also thank Hendrik for many
helpful discussions, and for his encouragement since the paper was first circulated in the Fall of 2009.
We are thankful to Florian Hess for valuable discussions. Finally, we thank the referees for their helpful comments.

\end{document}